\documentclass[11pt]{amsart}
\usepackage{hyperref}
\usepackage{marvosym}
\usepackage{fullpage}
\usepackage{graphics}
\usepackage{bbm}
\usepackage{amsmath}
\usepackage{amssymb}
\usepackage{dsfont}
\usepackage{diagbox}
\usepackage{enumitem}
\usepackage{caption, subcaption}
\usepackage{tikz,color}
\usepackage{faktor}
\usetikzlibrary{shapes.misc}

 \newtheorem{theorem}{Theorem}[section]
 \newtheorem{corollary}[theorem]{Corollary}
 
 \newtheorem{conjecture}[theorem]{Conjecture}
 \newtheorem{lemma}[theorem]{Lemma}

 \newtheorem{proposition}[theorem]{Proposition}
 
 \theoremstyle{definition}
 \newtheorem{example}[theorem]{Example}

 \newtheorem{definition}[theorem]{Definition}
 
 \newtheorem{remark}[theorem]{Remark}
 \newtheorem{algorithm}[theorem]{Algorithm}

\definecolor{mygreen}{rgb}{0.1,0.8,0.1}
\definecolor{mygray}{rgb}{0.7,0.7,0.7}

\newcommand{\tdot}[3]{\draw [fill=black,color=#3] (#1,#2) circle [radius=0.25];}

\newcommand{\arcdot}[3]{
    \draw [fill=black,color=#3] (#1,0) circle [radius=0.25];
    \node [above, yshift=2pt] at (#1,0) {{#2}};
}
\newcommand{\arcdraw}[2]{\draw [thick, out=-45, in=225] (#1,0) to (#2,0);}

\newcommand{\beq}{\begin{equation}}
\newcommand{\eeq}{\end{equation}}

\newcommand{\beqn}{\begin{eqnarray}}
\newcommand{\eeqn}{\end{eqnarray}}

\newcommand{\PF}[1]{\mathrm{PF}_{#1}}
\newcommand{\MVP}[1]{\mathrm{MVP}_{#1}}
\newcommand{\MotzPF}[1]{\mathrm{MotzPF}_{#1}}
\newcommand{\Motz}[1]{\mathrm{Motz}_{#1}}

\newcommand{\OOPF}[1]{\mathcal{O}_{\PF{#1}}}
\newcommand{\OOMVP}[1]{\mathcal{O}_{\MVP{#1}}}

\newcommand{\OPF}[2]{\mathcal{O}_{\PF{#1}}\left( #2 \right)}\newcommand{\OMVP}[2]{\mathcal{O}_{\MVP{#1}}\left( #2 \right)}

\newcommand{\FibMVP}[2]{\mathcal{O}^{-1}_{\MVP{#1}}\left( #2 \right)}

\newcommand{\Inv}[1]{\mathrm{Inv} \left( #1 \right)}
\newcommand{\LInv}[2]{\mathrm{LeftInv}_{#1} \left( #2 \right)}

\newcommand{\PFtoSub}{\Psi_{\mathrm{PF} \rightarrow \mathrm{Sub}}}
\newcommand{\SubtoPF}{\Psi_{\mathrm{Sub} \rightarrow \mathrm{PF}}}

\newcommand{\dec}[1]{\mathrm{dec}^{#1}}
\newcommand{\bipart}[2]{\mathrm{bipart}^{#1, #2}}
\newcommand{\splitright}[2]{\mathrm{split}^{#1, #2}}
\newcommand{\splitleft}[2]{\overline{\mathrm{split}}^{#1, #2}}

\newcommand{\Sub}[1]{\mathrm{Sub}^1\left( G_{#1} \right)}
\newcommand{\Valid}[1]{\mathrm{Valid}\left( G_{#1} \right)}

\newcommand{\NC}{\mathrm{NonCross}}

\newcommand{\Config}[1]{\mathrm{Config}_{#1}}
\newcommand{\Stable}[1]{\mathrm{Stable}_{#1}}
\newcommand{\minrec}[1]{\mathrm{minrec}\left( #1 \right)}
\newcommand{\Rec}[1]{\mathrm{Rec}_{#1}}
\newcommand{\MinRec}[1]{\mathrm{MinRec}_{#1}}

\newcommand{\CanTop}[1]{\mathrm{CanonTopp}\left( #1 \right)}

\newcommand{\Stab}{\mathrm{Stab}}

\newcommand{\Top}{\mathrm{Topp}}

\newcommand{\Zp}{\mathbb{Z}_+}

\begin{document}

\title{New combinatorial perspectives on MVP parking functions and their outcome map}
\author{Thomas Selig and Haoyue Zhu}
\address{Department of Computing, School of Advanced Technology, Xi'an Jiaotong-Livepool University \\
111, Ren'ai Road, Suzhou 215123, China} 
\date{\today}

\begin{abstract}
In parking problems, a given number of cars enter a one-way street sequentially, and try to park according to a specified preferred spot in the street. Various models are possible depending on the chosen rule for \emph{collisions}, when two cars have the same preferred spot. We study a model introduced by Harris, Kamau, Mori, and Tian in recent work, called the \emph{MVP parking problem}. In this model, priority is given to the cars arriving later in the sequence. When a car finds its preferred spot occupied by a previous car, it ``bumps'' that car out of the spot and parks there. The earlier car then has to drive on, and parks in the first available spot it can find. If all cars manage to park through this procedure, we say that the list of preferences is an MVP parking function.

We study the outcome map of MVP parking functions, which describes in what order the cars end up. In particular, we link the fibres of the outcome map to certain subgraphs of the inversion graph of the outcome permutation. This allows us to reinterpret and improve bounds from Harris \emph{et al.} on the fibre sizes. We then focus on a subset of parking functions, called \emph{Motzkin parking functions}, where every spot is preferred by at most two cars. We generalise results from Harris \emph{et al.}, and exhibit rich connections to Motzkin paths. We also give a closed enumerative formula for the number of MVP parking functions whose outcome is the complete bipartite permutation. Finally, we give a new interpretation of the MVP outcome map in terms of an algorithmic process on recurrent configurations of the Abelian sandpile model.
\end{abstract}

\maketitle

%MSC classifications: 05A19 (Primary), 05A05, 05A15, 05C20 (Secondary)

%Keywords: parking functions, permutation graphs, Motzkin paths, 

%%%%%%%%%%%%%%%%%%%%%%%% SECTION 1 %%%%%%%%%%%%%%%%%%%%%

\section{Introduction}\label{sec:intro}

In this section we introduce classical and MVP parking functions, and their outcome maps. Throughout the paper, $n$ represents a positive integer, and we denote $[n] := \{1, \cdots, n\}$.

\subsection{Parking functions and their variations}\label{subsec:PF_intro}

A \emph{parking preference} is a vector $p = (p_1, \cdots, p_n) \in [n]^n$. We think of $p_i$ as denoting the preferred parking spot of car $i$ in a car park with $n$ labelled spots. The car park is one-directional, with cars entering on the left in spot $1$ and driving through to spot $n$ (or until they park). Cars enter sequentially, in order $1, \cdots, n$. If the spot $p_i$ is unoccupied when car $i$ enters, it simply parks there. If this is not the case, then a previous car $j < i$ has already occupied spot $p_i$. We call this a \emph{collision} between cars $i$ and $j$.

In classical parking functions, such collisions are handled by giving priority to the earlier car $j$. This means that car $i$ is forced to drive on, and looks for the first unoccupied spot $k > p_i$. If no such spot exists, then car $i$ exits the car park, having failed to find a spot. We say that $p$ is a \emph{parking function} if all cars manage to park. Parking functions were originally introduced by Konheim and Weiss~\cite{KonWeiss} in their study of hashing functions. Since then, they have been a popular research topic in Mathematics and Computer Science, with rich connections to a variety of fields such as graph theory, representation theory, hyperplane arrangements, discrete geometry, and the Abelian sandpile model~\cite{ChevPyl, CoriPou, CR, DukesSplit, StanHyp, PitStan}. We refer the interested reader to the excellent survey by Yan~\cite{Yan}.

One may notice that the collision rule for parking functions has many possible variations, and indeed many variants of parking functions have been studied in the literature. 
\begin{itemize}
\item \textbf{Defective parking functions}~\cite{CamDefPF}. In this model, $m$ cars enter a one-way street with $n$ parking spots, and try to park following the classical parking rules. If $k$ cars are not able to park, we call the parking preference a \emph{defective parking function of defect $k$}. These correspond to classical parking functions when taking $m=n$ and $k=0$.
\item \textbf{Naples parking functions}~\cite{HarrisNaples1, HarrisNaples2}. In this model, when a car's preferred spot is occupied, it can first reverse up to some fixed number $k$ of spots to try to park, before driving on as in the classical parking function. These correspond to classical parking functions when $k=0$.
\item \textbf{Parking assortments and parking sequences}~\cite{AdeYan, HarrisSeq2, EhrHapp, HarrisSeq1}. In these models we have cars of different sizes, with just enough space in total for all cars to park. A car will try to park in the first available spot on or after its preference, but can only park there if there is sufficient space (i.e.\ the consecutive number of available spots at that location is greater than or equal to the car's size). In parking sequences, if this is not the case, the car immediately gives up and exits, whereas in parking assortments it will drive on and attempt to find a large enough space further along in the car park. Both models correspond to classical parking functions if all cars have size $1$.
\item \textbf{Vector parking functions}, or $\mathbf{u}$-parking functions~\cite{Yan2023, Yin2023}. Given a parking preference $p=(p_1,p_2,\cdots,p_n)$ and a vector $\mathbf{u}=(u_1,u_2,\cdots,u_n)$, we say that $p$ is a $\mathbf{u}$-parking function if its non-decreasing rearrangement $(a_1,a_2,\cdots,a_n)$ satisfies $a_i \leq u_i$ for all $i \in [n]$. These correspond to classical parking functions in the case where $\mathbf{u} = (1, 2, \cdots, n)$ (see e.g.~\cite[Section~1.1]{Yan} for this characterisation of classical parking functions).
\item \textbf{Graphical parking functions} on some graph $G$, or $G$-parking functions~\cite{PostPF}. In this model, cars try to park on vertices of a graph instead of in a one-way street. Like classical parking functions, $G$-parking functions are also connected with the Abelian sandpile model, via a bijection to its recurrent configurations~\cite[Lemma~13.6]{PostPF}.
\item \textbf{Higher-dimensional parking functions}. There are various models of these. The first, called \textbf{$(p,q)$-parking functions}, were introduced by Cori and Poulalhon~\cite{CoriPou}, in connection with the Abelian sandpile model on complete bipartite graphs with one extra distinguished root vertex. Dukes~\cite{DukesSplit} introduced a notion of \textbf{tiered parking functions}, which he connected to the Abelian sandpile model on complete split graphs. Both of these models involve cars of different colours or tiers, with extra conditions on where a car can park depending on its tier/colour. Motivated by previous work in this direction, Snider and Yan~\cite{YanUpq, Yan2023} defined a notion of \textbf{multidimensional $\mathbf{U}$-parking functions}, where $\mathbf{U} = \{(u_{i,j}, v_{i,j})_{1 \leq i \leq p, 1 \leq j \leq q} \}$ is a set of multidimensional vectors. These generalise both the $(p,q)$-parking functions of Cori and Poulalhon, and the (one-dimensional) vector parking functions discussed above (see also~\cite{Khare2014} for a beautiful connection to Gon{\v{c}}arov polynomials).
\end{itemize}

In this paper, we are interested in another variant called \emph{MVP parking functions}, introduced recently by Harris \emph{et al.}~\cite{HarrisMVP}. In this model, if there is a collision between two cars $j < i$, priority is given to the later car $i$. In other words, car $i$ will park in its preferred spot $p_i$. If that spot is already occupied by a previous car $j$, then car $j$ gets ``bumped'' out, and has to drive on. It then (re-)parks in the first available spot $k \geq p_i$. Note that bumpings do not propagate: the ``bumped'' car $j$ does not subsequently bump any other car. If all cars manage to park in this process, we say that $p$ is an MVP parking function.

It is in fact straightforward to check that a parking preference $p$ is an MVP parking function if, and only if, $p$ is a (classical) parking function. Indeed, in both MVP and classical processes, in determining whether all cars can park, the labels of the cars are unimportant: all that matters is which set of spots is occupied at any given time. We denote $\PF{n}$ or $\MVP{n}$ the set of parking functions of length $n$. These sets are the same due to the previous observation, but it will be convenient to use different notation depending on whether we are considering the classical or MVP parking process.

\subsection{The outcome maps}\label{subsec:outcome}

While the sets of MVP and classical parking functions are the same, these two processes differ in their \emph{outcome map}. This map describes where the cars end up. More precisely, if $p$ is a parking function, its \emph{outcome} is a permutation $\pi = \pi_1 \cdots \pi_n$, where for all $i \in [n]$, $\pi_i$ is the label of the car occupying spot $i$ when all cars have parked. The classical, resp.\ MVP, outcome map, denoted $\OOPF{n}$, resp.\ $\OOMVP{n}$, is then the map $p \mapsto \pi$ describing the outcome of the classical, resp.\ MVP, parking process. The following example illustrates the classical and MVP parking processes, and shows how their outcomes may differ.

\begin{example}\label{ex:different_outcomes}
Consider the parking function $p = (3, 1, 1, 2)$. Under the classical parking process in Figure~\ref{fig:exa_OPF}, car $1$ first parks in spot $3$, followed by car $2$ parking in spot $1$. Then car $3$ wishes to park in spot $1$ but cannot do so, so it drives on and parks in spot $2$ (the first available spot at this point). Finally, car $4$ wishes to park in spot $2$. However, $2$ is occupied, so car $4$ drives on: $3$ is also occupied (by car $1$), so car $4$ ends up parking in spot $4$. Finally, we get the outcome $\pi = \OPF{4}{p} = 2314$.

\begin{figure}[ht]

 \centering
 
  \begin{tikzpicture}[scale=0.2]
    
    \node at (-3,1.2) {cars};
    \node at (-3,-4.5) {spots};       
    
    %%Car 1
    \foreach \x in {1,...,4}
	  \node at (-1+4*\x,-4.5) {$\x$};
    %Parking spots
    \draw [thick, color=blue] (1,-2)--(1,-3)--(5,-3)--(5,-2);
    \draw [thick, color=blue] (5,-2)--(5,-3)--(9,-3)--(9,-2);
    \draw [thick, color=blue] (9,-2)--(9,-3)--(13,-3)--(13,-2);
    \draw [thick, color=blue] (13,-2)--(13,-3)--(17,-3)--(17,-2);
    \node [draw, circle, color=red, scale=0.8pt] (1) at (11,1.2) {$1$};
    \node at (11, -1.8) {$\downarrow$};
    \node at (21, 0) {$\Longrightarrow$};

    %%Car 2
    \begin{scope}[shift={(24,0)}]
    \foreach \x in {1,...,4}
	  \node at (-1+4*\x,-4.5) {$\x$};
	%Parking spots
    \draw [thick, color=blue] (1,-2)--(1,-3)--(5,-3)--(5,-2);
    \draw [thick, color=blue] (5,-2)--(5,-3)--(9,-3)--(9,-2);
    \draw [thick, color=blue] (9,-2)--(9,-3)--(13,-3)--(13,-2);
    \draw [thick, color=blue] (13,-2)--(13,-3)--(17,-3)--(17,-2);
    \node [draw, circle, scale=0.8pt] (1) at (11,-1) {$1$};
    \node [draw, circle, color=red, scale=0.8pt] (1) at (3,1.2) {$2$};
    \node at (3, -1.8) {$\downarrow$};
    \node at (21, 0) {$\Longrightarrow$};
    \end{scope}

    %%Car 3
    \begin{scope}[shift={(48,0)}]
    \foreach \x in {1,...,4}
	  \node at (-1+4*\x,-4.5) {$\x$};
	%Parking spots
    \draw [thick, color=blue] (1,-2)--(1,-3)--(5,-3)--(5,-2);
    \draw [thick, color=blue] (5,-2)--(5,-3)--(9,-3)--(9,-2);
    \draw [thick, color=blue] (9,-2)--(9,-3)--(13,-3)--(13,-2);
    \draw [thick, color=blue] (13,-2)--(13,-3)--(17,-3)--(17,-2);
    \node [draw, circle, scale=0.8pt] (1) at (11,-1) {$1$};
    \node [draw, circle, scale=0.8pt] (2) at (3,-1) {$2$};
    \node [draw, circle, color=red, scale=0.8pt] (3) at (3,3) {$3$};
    \node at (6.2, 3) {$\curvearrowright$};
    \end{scope}

    %%Car 4
    \begin{scope}[shift={(0,-12)}]
    \node at (-3, 0) {$\Longrightarrow$};
    \foreach \x in {1,...,4}
	  \node at (-1+4*\x,-4.5) {$\x$};
	%Parking spots
    \draw [thick, color=blue] (1,-2)--(1,-3)--(5,-3)--(5,-2);
    \draw [thick, color=blue] (5,-2)--(5,-3)--(9,-3)--(9,-2);
    \draw [thick, color=blue] (9,-2)--(9,-3)--(13,-3)--(13,-2);
    \draw [thick, color=blue] (13,-2)--(13,-3)--(17,-3)--(17,-2);
    \node [draw, circle, scale=0.8pt] (1) at (11,-1) {$1$};
    \node [draw, circle, scale=0.8pt] (2) at (3,-1) {$2$};
    \node [draw, circle, scale=0.8pt] (3) at (7,-1) {$3$};
    \node [draw, circle, color=red, scale=0.8pt] (4) at (7,3) {$4$};
    \node at (11, 3) {$\longrightarrow$};
    \node at (14.5, 2) {$\searrow$};
    \node at (21, 0) {$\Longrightarrow$};
    \end{scope}

    %%Final config
    \begin{scope}[shift={(24,-12)}]
    \foreach \x in {1,...,4}
	  \node at (-1+4*\x,-4.5) {$\x$};
	%Parking spots
    \draw [thick, color=blue] (1,-2)--(1,-3)--(5,-3)--(5,-2);
    \draw [thick, color=blue] (5,-2)--(5,-3)--(9,-3)--(9,-2);
    \draw [thick, color=blue] (9,-2)--(9,-3)--(13,-3)--(13,-2);
    \draw [thick, color=blue] (13,-2)--(13,-3)--(17,-3)--(17,-2);
    \node [draw, circle, scale=0.8pt] (1) at (11,-1) {$1$};
    \node [draw, circle, scale=0.8pt] (2) at (3,-1) {$2$};
    \node [draw, circle, scale=0.8pt] (3) at (7,-1) {$3$};
    \node [draw, circle, scale=0.8pt] (4) at (15,-1) {$4$};
    \end{scope}
    
  \end{tikzpicture}
  \caption{The classical parking process with $p=(3,1,1,2)$ and $\OPF{4}{p}=2314$.\label{fig:exa_OPF}}
\end{figure}
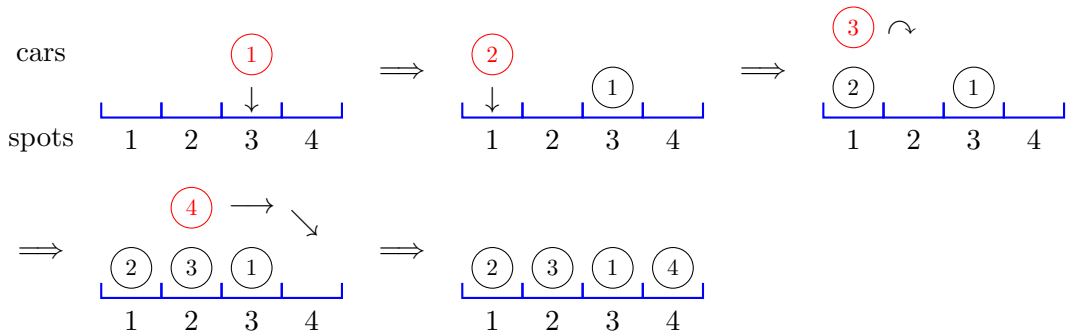

Now consider the same parking function $p$, but for the MVP parking process. In Figure~\ref{fig:exa_OMVP}, again,  cars $1$ and $2$ park in spots $3$ and $1$ respectively. Now car $3$ arrives, and sees car $2$ in its preferred spot (spot $1$). It bumps car $2$ out of spot $1$, forcing it to drive on. Spot $2$ is available, so car $2$ parks there. Finally car $4$ arrives and sees car $2$ in its preferred spot (spot $2$). It bumps car $2$, forcing it to drive on and park in the only remaining spot, which is spot $4$. Finally, we get the outcome $\pi = \OMVP{4}{p} = 3412$.

\begin{figure}[ht]

 \centering
 
  \begin{tikzpicture}[scale=0.2]
    
    \node at (-3,1.2) {cars};
    \node at (-3,-4.5) {spots};       
    
    %%Car 1
    \foreach \x in {1,...,4}
	  \node at (-1+4*\x,-4.5) {$\x$};
    %Parking spots
    \draw [thick, color=blue] (1,-2)--(1,-3)--(5,-3)--(5,-2);
    \draw [thick, color=blue] (5,-2)--(5,-3)--(9,-3)--(9,-2);
    \draw [thick, color=blue] (9,-2)--(9,-3)--(13,-3)--(13,-2);
    \draw [thick, color=blue] (13,-2)--(13,-3)--(17,-3)--(17,-2);
    \node [draw, circle, color=red, scale=0.8pt] (1) at (11,1.2) {$1$};
    \node at (11, -1.8) {$\downarrow$};
    \node at (21, 0) {$\Longrightarrow$};

    %%Car 2
    \begin{scope}[shift={(24,0)}]
    \foreach \x in {1,...,4}
	  \node at (-1+4*\x,-4.5) {$\x$};
	%Parking spots
    \draw [thick, color=blue] (1,-2)--(1,-3)--(5,-3)--(5,-2);
    \draw [thick, color=blue] (5,-2)--(5,-3)--(9,-3)--(9,-2);
    \draw [thick, color=blue] (9,-2)--(9,-3)--(13,-3)--(13,-2);
    \draw [thick, color=blue] (13,-2)--(13,-3)--(17,-3)--(17,-2);
    \node [draw, circle, scale=0.8pt] (1) at (11,-1) {$1$};
    \node [draw, circle, color=red, scale=0.8pt] (1) at (3,1.2) {$2$};
    \node at (3, -1.8) {$\downarrow$};
    \node at (21, 0) {$\Longrightarrow$};
    \end{scope}

    %%Car 3
    \begin{scope}[shift={(48,0)}]
    \foreach \x in {1,...,4}
	  \node at (-1+4*\x,-4.5) {$\x$};
	%Parking spots
    \draw [thick, color=blue] (1,-2)--(1,-3)--(5,-3)--(5,-2);
    \draw [thick, color=blue] (5,-2)--(5,-3)--(9,-3)--(9,-2);
    \draw [thick, color=blue] (9,-2)--(9,-3)--(13,-3)--(13,-2);
    \draw [thick, color=blue] (13,-2)--(13,-3)--(17,-3)--(17,-2);
    \node [draw, circle, scale=0.8pt] (1) at (11,-1) {$1$};
    \node [draw, circle, scale=0.8pt] (2) at (3,-1) {$2$};
    \node [draw, circle, color=red, scale=0.8pt] (3) at (3,3) {$3$};
    \node at (6.2, -0.5) {$\curvearrowright$};
    \end{scope}

    %%Car 4
    \begin{scope}[shift={(0,-12)}]
    \node at (-3, 0) {$\Longrightarrow$};
    \foreach \x in {1,...,4}
	  \node at (-1+4*\x,-4.5) {$\x$};
	%Parking spots
    \draw [thick, color=blue] (1,-2)--(1,-3)--(5,-3)--(5,-2);
    \draw [thick, color=blue] (5,-2)--(5,-3)--(9,-3)--(9,-2);
    \draw [thick, color=blue] (9,-2)--(9,-3)--(13,-3)--(13,-2);
    \draw [thick, color=blue] (13,-2)--(13,-3)--(17,-3)--(17,-2);
    \node [draw, circle, scale=0.8pt] (1) at (11,-1) {$1$};
    \node [draw, circle, scale=0.8pt] (2) at (7,-1) {$2$};
    \node [draw, circle, scale=0.8pt] (3) at (3,-1) {$3$};
    \node [draw, circle, color=red, scale=0.8pt] (4) at (7,3) {$4$};
    \draw [->, out=30, in=150] (8.2,0.8) to (11.5,1);
    \draw [->, out=0, in=135] (12,1) to (15.5,-0.5);
    \node at (21, 0) {$\Longrightarrow$};
    \end{scope}

    %%Final config
    \begin{scope}[shift={(24,-12)}]
    \foreach \x in {1,...,4}
	  \node at (-1+4*\x,-4.5) {$\x$};
	%Parking spots
    \draw [thick, color=blue] (1,-2)--(1,-3)--(5,-3)--(5,-2);
    \draw [thick, color=blue] (5,-2)--(5,-3)--(9,-3)--(9,-2);
    \draw [thick, color=blue] (9,-2)--(9,-3)--(13,-3)--(13,-2);
    \draw [thick, color=blue] (13,-2)--(13,-3)--(17,-3)--(17,-2);
    \node [draw, circle, scale=0.8pt] (1) at (11,-1) {$1$};
    \node [draw, circle, scale=0.8pt] (2) at (15,-1) {$2$};
    \node [draw, circle, scale=0.8pt] (3) at (3,-1) {$3$};
    \node [draw, circle, scale=0.8pt] (4) at (7,-1) {$4$};
    \end{scope}
    
  \end{tikzpicture}
  \caption{The MVP parking process with $p=(3,1,1,2)$ and $\OMVP{4}{p}=3412$.\label{fig:exa_OMVP}}
\end{figure}
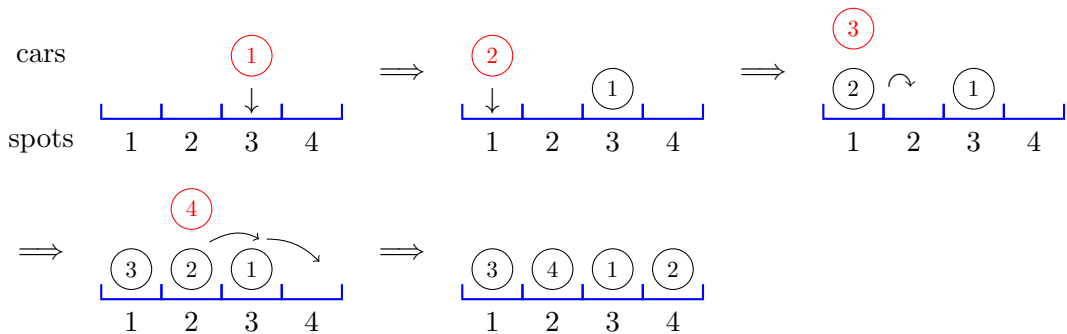

\end{example}

\subsection{Content of the paper}\label{subsec:organisation}

In this paper, our main goal is to study the \emph{fibres} of the MVP outcome map. That is, for a given, fixed permutation $\pi \in S_n$, we are interested in the set $\FibMVP{n}{\pi}$ of parking functions whose (MVP) outcome is the permutation $\pi$. The paper is organised as follows. 

In Section~\ref{sec:general_case}, we provide an interpretation of the MVP outcome fibre set $\FibMVP{}{\pi}$ in terms of certain subgraphs, called $1$-subgraphs, of the corresponding permutation inversion graph $G_{\pi}$ (see Section~\ref{subsec:inversion_graph} for a definition of these). We describe mappings from parking functions to $1$-subgraphs (Definition~\ref{def:PFtoSub}) and from $1$-subgraphs to parking functions (Definition~\ref{def:SubtoPF}), and show that these are inverse to each other when restricted to so-called \emph{valid} $1$-subgraphs (Theorem~\ref{thm:MVPPF_subgraphs}). We use this characterisation in Section~\ref{subsec:bounds} to give improved upper (Proposition~\ref{pro:upper_bound}) and lower (Proposition~\ref{pro:lower_bound}) bounds on the fibre sizes, by providing necessary or sufficient conditions for a $1$-subgraph to be valid.

Section~\ref{sec:MotzPF} is dedicated to a study of a certain subset of MVP parking functions in which each parking space is preferred by at most two cars. We call these \emph{Motzkin} parking functions due to a rich connection to Motzkin paths (see Theorem~\ref{thm:MotzPF_characterisation} and Theorem~\ref{thm:bij_MotzPFequiv_Motz}). These results generalise a bijection between Motzkin paths and parking functions whose MVP outcome is the decreasing permutation $\dec{n}:=n(n-1)\cdots1$ established by Harris \emph{et al.}~\cite[Theorem~4.2]{HarrisMVP}. Inspired by this, we apply the constructions of Section~\ref{sec:general_case} to obtain a bijection between the decreasing fibre set $\FibMVP{n}{\dec{n}}$ and the set of non-crossing matching arc-diagrams on $n$ vertices (Theorem~\ref{thm:nonCrossing_dec}).
In Section~\ref{sec:m_n} we study another family of MVP parking function, namely those whose outcome permutation $\bipart{m}{n}$ corresponds to the complete bipartite graph $K_{m,n}$, i.e.\ $\bipart{m}{n}:=(n+1)(n+2)\cdots (n+m) 12 \cdots n$. We give a closed formula enumerating the MVP outcome fibre in this case when $n=2$ (Theorem~\ref{thm:enum_m_2}).

Section~\ref{sec:asm_prelims} connects MVP parking functions with recurrent configurations of the \emph{Abelian sandpile model} (ASM) on complete graphs (Theorem~\ref{thm:bij_rec_pf}). We provide an algorithmic process (Algorithm~\ref{algo:c-minrec}) to track the bumpings that occur in the MVP parking process through the corresponding recurrent configuration. This then allows us to interpret the MVP outcome map by combining this process with a classical notion of ``canonical toppling'' in the ASM (Theorem~\ref{thm:minrec_outcome_CanonTopp}). 
Finally, Section~\ref{sec:future} summarises our results and discusses some possible future research directions.

%%%%%%%%%%%%%%%%%%%%%%%% SECTION 2 %%%%%%%%%%%%%%%%%%%%%

\section{General case}\label{sec:general_case}

In this section we study the MVP outcome map in the general setting. We will give an interpretation of the fibres in terms of certain subgraphs of the inversion graph of the outcome permutation.

\subsection{Inversion graphs and subgraphs}\label{subsec:inversion_graph}

Given a permutation $\pi = \pi_1 \cdots \pi_n \in S_n$, we say that a pair $(j, i)$ is an \emph{inversion} of $\pi$ if $j < i$ and $\pi_j > \pi_i$. We denote $\Inv{\pi}$ the set of inversions of $\pi$. For any $i \in [n]$ we define the set of \emph{left-inversions} at $i$ in $\pi$ by $\LInv{\pi}{i} := \{ j \in [n]; \, (j,i) \in \Inv{\pi}\}$. The \emph{inversion graph} of a permutation $\pi$, denoted $G_{\pi}$, is the graph with vertex set $[n]$ and edge set $\Inv{\pi}$.

It will be convenient to represent permutations and their inversion graphs graphically in a $n \times n$ grid. We label rows and columns $1, \cdots, n$ from top to bottom and left to right respectively. The graphical representation of a permutation $\pi$ consists in placing a dot in each row $\pi_i$ and column $i$. The edges of the corresponding inversion graph are then pairs of dots where one is above and to the right of the other. We may sometimes think of edges $(j, i)$ with $j < i$ as directed from $j$ to $i$ (i.e.\ from left to right), and refer to them as \emph{arcs}.

\begin{example}\label{ex:inversions}
Consider the permutation $\pi = 42315$. The inversions are the pairs of indices $(1,2)$, $(1, 3)$, $(1, 4)$, $(2, 4)$ and $(3, 4)$. Figure~\ref{fig:inversion_graph} shows the graphical representations of $\pi$ and of its inversion graph.

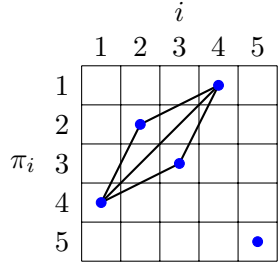
\begin{figure}[ht]
\centering
\begin{tikzpicture}[scale=0.26]
%inversion graph
\draw [step=2] (2,2) grid (12,-8);
\node at (-1,-3) {$\pi_i$};
\node at (7, 4.8) {$i$};
%row labels
\foreach \x in {1,...,5}
  \node at (1+2*\x, 3) {$\x$};
%column labels
\foreach \y in {1,...,5}
  \node at (1, 3-2*\y) {$\y$};
%tree edges
\draw [thick] (3,-5)--(5,-1);
\draw [thick] (3,-5)--(7,-3);
\draw [thick] (5,-1)--(9,1);
\draw [thick] (3,-5)--(9,1);
\draw [thick] (7,-3)--(9,1);
%leaves
\tdot{3}{-5}{blue}
\tdot{5}{-1}{blue}
\tdot{7}{-3}{blue}
\tdot{9}{1}{blue}
\tdot{11}{-7}{blue}
\end{tikzpicture}
\caption{The permutation $\pi=42315$ and its inversion graph $G_{\pi}$. \label{fig:inversion_graph}}
\end{figure}
\end{example}

We will use certain subgraphs of inversion graphs to represent MVP parking functions. Here, subgraphs are considered to be vertex-spanning, so that a subgraph is simply a subset of edges of the original graph.
For a permutation $\pi$ and corresponding inversion graph $G_{\pi}$, we define $\Sub{\pi} := \{ S \subseteq \Inv{\pi}; \ \forall i \in [n],\, \vert \{ j \in [n];\, (j,i) \in S \} \vert \leq 1  \}$. 
In words, this is the set of subgraphs of $G_{\pi}$ where the number of incident \emph{left-arcs} at any vertex is at most $1$. We refer to elements of $\Sub{\pi}$ as \emph{$1$-subgraphs} of $G_{\pi}$. Figure~\ref{fig:1-subgraphs} below shows the three $1$-subgraphs of $G_{231}$ in the three left-most graphs. The right-most graph (crossed-out) is not a $1$-subgraph, since the vertex in column $i = 3$ has two incident left-arcs.

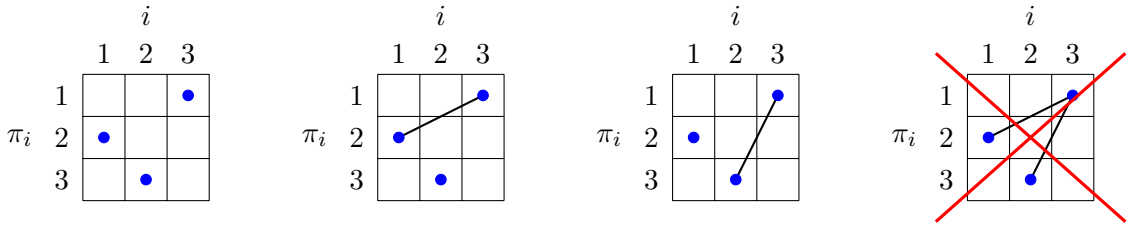
\begin{figure}[ht]

\centering

\begin{tikzpicture}[scale=0.28]
%inversion graph
\draw [step=2] (2,2) grid (8,-4);
\node at (-1, -1) {$\pi_i$};
\node at (5, 4.8) {$i$};
%row labels
\foreach \x in {1,...,3}
  \node at (1+2*\x, 3) {$\x$};
%column labels
\foreach \y in {1,...,3}
  \node at (1, 3-2*\y) {$\y$};
%dots
\tdot{3}{-1}{blue}
\tdot{5}{-3}{blue}
\tdot{7}{1}{blue}

%first subgraph
\begin{scope}[shift={(14,0)}]
%inversion graph
\draw [step=2] (2,2) grid (8,-4);
\node at (-1, -1) {$\pi_i$};
\node at (5, 4.8) {$i$};
%row labels
\foreach \x in {1,...,3}
  \node at (1+2*\x, 3) {$\x$};
%column labels
\foreach \y in {1,...,3}
  \node at (1, 3-2*\y) {$\y$};
%edges
\draw [thick] (3,-1)--(7,1);
%dots
\tdot{3}{-1}{blue}
\tdot{5}{-3}{blue}
\tdot{7}{1}{blue}
\end{scope}

%second subgraph
\begin{scope}[shift={(28,0)}]
%inversion graph
\draw [step=2] (2,2) grid (8,-4);
\node at (-1, -1) {$\pi_i$};
\node at (5, 4.8) {$i$};
%row labels
\foreach \x in {1,...,3}
  \node at (1+2*\x, 3) {$\x$};
%column labels
\foreach \y in {1,...,3}
  \node at (1, 3-2*\y) {$\y$};
%edges
\draw [thick] (5,-3)--(7,1);
%dots
\tdot{3}{-1}{blue}
\tdot{5}{-3}{blue}
\tdot{7}{1}{blue}
\end{scope}

%third subgraph, not 1-subgraph
\begin{scope}[shift={(42,0)}]
%inversion graph
\draw [step=2] (2,2) grid (8,-4);
\node at (-1, -1) {$\pi_i$};
\node at (5, 4.8) {$i$};
%row labels
\foreach \x in {1,...,3}
  \node at (1+2*\x, 3) {$\x$};
%column labels
\foreach \y in {1,...,3}
  \node at (1, 3-2*\y) {$\y$};
%edges
\draw [thick] (3,-1)--(7,1);
\draw [thick] (5,-3)--(7,1);
%dots
\tdot{3}{-1}{blue}
\tdot{5}{-3}{blue}
\tdot{7}{1}{blue}
%cross out
\draw[very thick, red] (0.5,-5)--(9.5,3);
\draw[very thick, red] (0.5,3)--(9.5,-5);
\end{scope}
\end{tikzpicture}

\caption{The three $1$-subgraphs of the inversion graph $G_{231}$; the right-most is not a $1$-subgraph as there are two left edges incident to $i = 3$.\label{fig:1-subgraphs}}

\end{figure}

\subsection{Connecting $1$-subgraphs to the MVP outcome map}\label{subsec:MVP_outcome_subgraphs}

In this section, we explain how to represent parking functions in the MVP outcome fibre of a given permutation $\pi$ via $1$-subgraphs of the permutation's inversion graph, and vice versa.

\begin{definition}\label{def:PFtoSub}
Let $\pi \in S_n$ be a permutation. We define a map $\PFtoSub: \FibMVP{n}{\pi} \rightarrow \Sub{\pi}$, $p \mapsto S(p)$ as follows:
\beq\label{eq:PFtoSub}
S(p) := \{ (j,i) \in \Inv{\pi}; \, p_{\pi_i} = j \}.
\eeq
\end{definition}

In words, if the car $\pi_i$ that ends up in spot $i$ initially preferred some spot $j < i$ in the parking function $p$ (so was eventually bumped to $i$ in the MVP parking process), then we put an edge from $j$ to $i$ in $S(p)$. Note that for this bumping to occur, the car $\pi_j$ which eventually ends up in spot $j$ must enter the car park after car $\pi_i$, which exactly means that $(j,i)$ is an inversion of $\pi$ (i.e.\ an edge of $G_{\pi}$). Moreover, since exactly one car ends up in any given spot $i$, there is at most one left-arc incident to $i$ in $S(p)$ (in the case where $p_{\pi_i} = i$, i.e.\ the car that ends up in $i$ wanted to park there, we have no incident left-arc), so that $S(p)$ is indeed a $1$-subgraph of $G_{\pi}$, as desired. We can then define an inverse for this map from parking functions to sub-graphs, as follows.

\begin{definition}\label{def:SubtoPF}
Let $\pi \in S_n$ be a permutation. We define a map $\SubtoPF: \Sub{\pi} \rightarrow \MVP{n}$, $S \mapsto p = p(S)$ as follows:
\beq\label{eq:SubtoPF}
p_{\pi_i}=
\begin{cases}
i \qquad \text{ if }\quad \vert \{ j \in [n]; (j,i) \in S \} \vert =0, \\
j \qquad \text{ if }\quad j\text{ is the unique } j<i \text{ such that }(j,i)\in S.
\end{cases}
\eeq
\end{definition}

In words, if there is no left-arc incident to $i$ in $S$, we set $p_{\pi_i} = i$. Otherwise, since $S$ is a $1$-subgraph, there is a unique left-arc $(j,i)$ incident to $i$ in $S$, and we set $p_{\pi_i} = j$. 

\begin{example}\label{ex:SubtoPF}
Consider the permutation $\pi = 34125$ and the $1$-subgraph $S \in \Sub{\pi}$ consisting of the arcs $(2,3)$ and $(2,4)$ as in Figure~\ref{fig:example_SubtoPF}. We calculate $p := \SubtoPF(S)$ as follows. First, let us determine $p_1$ the preference of car $1$. Note that $1 = \pi_3$ here, so we are looking at the vertex in row $1$, column $3$ (labelled $3$ in our inversion graph labelling). Here there is a left-arc incident to this vertex, whose left end-point is in column $2$, yielding $p_1 = 2$. Similarly, $p_2 = 2$ also, since there is a left arc incident to the dot in row $2$, column $4$, whose left-end point is also in column $2$. However, the dot in row $3$, column $1$, has no incident left-arc, and neither does the dot in row $4$, column $2$, or the dot in row $5$, column $5$. We therefore set $p_3 = 1$, $p_4 = 2$, and $p_5$ = $5$. Finally, we get the preference $p(S)=(p_1, p_2, p_3, p_4, p_5) = (2,2,1,2,5)$.

\begin{figure}[ht]
\centering
\begin{tikzpicture}[scale=0.255]
%inversion graph
\draw [step=2] (2,2) grid (12,-8);
\node at (-1, -1.2) {$\pi_i$};
\node at (-1, -3) {$\downarrow$};
\node at (-1.5, -4.8) {cars};
\node at (7, 4.8) {$i \rightarrow$ spots};
%row labels
\foreach \x in {1,...,5}
  \node at (1+2*\x, 3) {$\x$};
%column labels
\foreach \y in {1,...,5}
  \node at (1, 3-2*\y) {$\y$};
%tree edges
\draw [thick] (5,-5)--(7,1);
\draw [thick] (5,-5)--(9,-1);
%leaves
\tdot{3}{-3}{blue}
\tdot{5}{-5}{blue}
\tdot{7}{1}{blue}
\tdot{9}{-1}{blue}
\tdot{11}{-7}{blue}
\end{tikzpicture}
\caption{A $1$-subgraph $S$ of $G_{34125}$; the corresponding parking function is $\SubtoPF(S) = (2,2,1,2,5)$.\label{fig:example_SubtoPF}}
\end{figure}
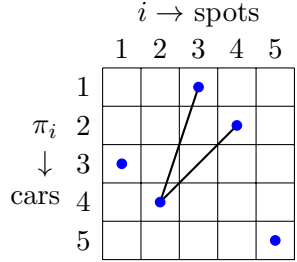

\end{example}

Our main result of this Section~\ref{sec:general_case} is the following.

\begin{theorem}\label{thm:MVPPF_subgraphs}
Let $\pi \in S_n$ be a permutation. The maps $\PFtoSub: \FibMVP{n}{\pi} \rightarrow \Sub{\pi}$ and $\SubtoPF: \Sub{\pi} \rightarrow \MVP{n}$ are injective. Moreover, for any $p \in \FibMVP{n}{\pi}$, we have $\SubtoPF(\PFtoSub(p))=p$.
\end{theorem}

\begin{proof}
We first show that for any $p \in \FibMVP{n}{\pi}$, we have $\SubtoPF(\PFtoSub(p))=p$. This follows essentially from the constructions in Definitions~\ref{def:PFtoSub} and \ref{def:SubtoPF}. Let $p \in \FibMVP{n}{\pi}$, and define $S := \PFtoSub(p)$ and $p' := \SubtoPF(S)$. We wish to show that $p' = p$. Fix some spot $i \in [n]$, and consider the car $\pi_i$ which ends up in spot $i$ in the MVP parking process for $p$. There are two cases to consider.

\textbf{Case 1:} car $\pi_i$ ended up in its preferred spot. By definition, this means that $p_{\pi_i} = i$. Moreover, in this case, by construction there is no left edge incident to $i$ in $S$ (the car ending up in spot $i$ originally preferred that spot). Therefore, by Definition~\ref{def:SubtoPF}, we have $p'_{\pi_i} = i = p_{\pi_i}$, as desired.

\textbf{Case 2:} car $\pi_i$ had initially preferred some spot $j < i$, and finally ended up in $i$ after (possibly multiple) bumpings. Then by definition we have $p_{\pi_i} = j$. Moreover, in this case, by construction we put an edge $(j,i)$ in $S$. Finally, by Definition~\ref{def:SubtoPF}, we have $p'_{\pi_i} = j = p_{\pi_i}$, as desired.

The equality $\SubtoPF(\PFtoSub(p))=p$ for all $p \in \FibMVP{n}{\pi}$ further implies the injectivity of the map $\PFtoSub$ on $\FibMVP{n}{\pi}$ since it has an inverse. It remains to be shown that $\SubtoPF$ is injective on its domain $\Sub{\pi}$. For this, fix $S, S' \in \Sub{\pi}$ two $1$-subgraphs such that $S \neq S'$. Define $p := \SubtoPF(S)$ and $p' := \SubtoPF(S')$. We wish to show that $p \neq p'$. Since $S \neq S'$, we may assume that there exists some edge $(j,i)$ with $j<i$ which is in $S$ but not in $S'$. By construction we have $p_{\pi_i} = j$ in this case. Moreover, depending on whether $i$ has an incident left edge in $S'$, say to $j' \neq j$, or whether $i$ has no incident edge, we will have either $p'_{\pi_i} = j'$ or $p'_{\pi_i} = i$. In both cases $p'_{\pi_i} \neq p_{\pi_i}$, and thus $p' \neq p$, as desired.
\end{proof}

As such, for a given permutation $\pi$, the map $\PFtoSub$ induces a bijection from the fibre $\FibMVP{n}{\pi}$ unto its image $\PFtoSub \left( \FibMVP{n}{\pi} \right)$. The question of calculating the fibre $\FibMVP{n}{\pi}$ then becomes that of calculating the image set, or, equivalently, calculating the set of $1$-subgraphs $S$ of $G_{\pi}$  such that $\OMVP{n}{\SubtoPF(S) }= \pi$.

\begin{definition}\label{def:valid_subgraph}
We say that a $1$-subgraph $S \in \Sub{\pi}$ is \emph{valid} if $\OMVP{n}{\SubtoPF(S) }= \pi$. Otherwise, we say that $S$ is \emph{invalid}. We denote $\Valid{\pi}$ the set of valid $1$-subgraphs of $G_{\pi}$.
\end{definition}

With this terminology, Theorem~\ref{thm:MVPPF_subgraphs} states that the map $\PFtoSub$ is a bijection from the fibre set $\FibMVP{n}{\pi}$ to the set $\Valid{\pi}$ of valid $1$-subgraphs of $G_{\pi}$. In particular, we get the following enumeration and upper bound.

\begin{corollary}\label{cor:Harris_bound}
For any permutation $\pi \in S_n$, we have 
$$\left\vert \FibMVP{n}{\pi} \right\vert = \left\vert \Valid{\pi} \right\vert \leq \left\vert \Sub{\pi} \right\vert = \prod\limits_{i \in [n]} \big( 1 + \vert \LInv{\pi}{i} \vert \big).$$
\end{corollary}

In fact, the upper bound of Corollary~\ref{cor:Harris_bound} was already established by Harris \emph{et al.}~\cite[Theorem~3.1]{HarrisMVP}, although its formulation in terms of inversions is new (in the cited work, it is stated in terms of cars/spots). In the same paper (\cite[Theorem~3.2]{HarrisMVP}), the authors also determined a full characterisation of when this upper bound is tight, which we again reformulate in the subgraph context. This characterisation is in terms of permutation \emph{patterns}.

Let $\pi \in S_n$, $\tau \in S_k$ be two permutations with $k \leq n$. We say that $\pi$ \emph{contains} the \emph{pattern} $\tau$ if there exist indices $i_1 < i_2 < \cdots < i_k$ such that $\pi_{i_1}, \pi_{i_2}, \cdots, \pi_{i_k}$ appear in the same relative order as $\tau$. If $\pi$ does not contain the pattern $\tau$, we say that $\pi$ \emph{avoids} $\tau$. For example, the permutation $\pi = 561243$ contains two occurrences of the pattern $321$ (in \textbf{bold}): $\mathbf{5}612\mathbf{43}$, $5\mathbf{6}12\mathbf{43}$. However, it avoids the pattern $4321$ since there is no sequence of $4$ numbers in decreasing order. The inversion graph $G_{\pi}$ of a permutation $\pi$ is acyclic if, and only if, $\pi$ avoids the patterns $321$ and $3412$, which correspond to cycles of length $3$ and $4$ respectively (see e.g.~\cite{BV}).

\begin{theorem}\label{thm:MVPPF_perm_patterns}
Let $\pi \in S_n$ be a permutation. Then all $1$-subgraphs of $G_{\pi}$ are valid if, and only if, $\pi$ avoids the patterns $321$ and $3412$, or equivalently if the graph $G_{\pi}$ is acyclic. In particular, in that case, we have $\left\vert \FibMVP{n}{\pi} \right\vert = \prod\limits_{i \in [n]} \big( 1 + \vert \LInv{\pi}{i} \vert \big)$.
\end{theorem}

\begin{example}\label{ex:valid}
Consider the permutation $\pi = 312$. The four $1$-subgraphs of $G_{\pi}$ are shown in Figure~\ref{fig:all_subgraphs}. Since $G_{\pi}$ is acyclic, all four of these are valid. Therefore to calculate the fibre set $\FibMVP{3}{312}$ we simply apply the map $\SubtoPF$ to each subgraph in turn. Finally, we get $\FibMVP{3}{312} = \{ (2,3,1), (1,3,1), (2,1,1), (1,1,1) \}$.

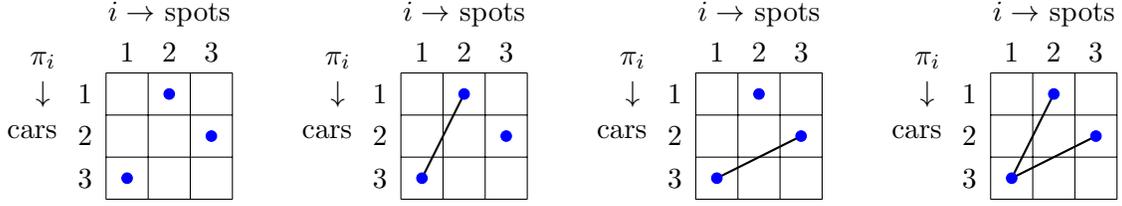
\begin{figure}[ht]

\centering

\begin{tikzpicture}[scale=0.28]
%inversion graph
\draw [step=2] (2,2) grid (8,-4);
\node at (-1, 2.6) {$\pi_i$};
\node at (-1, 1) {$\downarrow$};
\node at (-1.5, -0.8) {\text{cars}};
\node at (5, 4.8) {$i \rightarrow$ spots};
%row labels
\foreach \x in {1,...,3}
  \node at (1+2*\x, 3) {$\x$};
%column labels
\foreach \y in {1,...,3}
  \node at (1, 3-2*\y) {$\y$};
%leaves
\tdot{3}{-3}{blue}
\tdot{5}{1}{blue}
\tdot{7}{-1}{blue}

%first subgraph
\begin{scope}[shift={(14,0)}]
%inversion graph
\draw [step=2] (2,2) grid (8,-4);
\node at (-1, 2.6) {$\pi_i$};
\node at (-1, 1) {$\downarrow$};
\node at (-1.5, -0.8) {\text{cars}};
\node at (5, 4.8) {$i \rightarrow$ spots};
%row labels
\foreach \x in {1,...,3}
  \node at (1+2*\x, 3) {$\x$};
%column labels
\foreach \y in {1,...,3}
  \node at (1, 3-2*\y) {$\y$};
%tree edges
\draw [thick] (3,-3)--(5,1);
%leaves
\tdot{3}{-3}{blue}
\tdot{5}{1}{blue}
\tdot{7}{-1}{blue}
\end{scope}

%second subgraph
\begin{scope}[shift={(28,0)}]
%inversion graph
\draw [step=2] (2,2) grid (8,-4);
\node at (-1, 2.6) {$\pi_i$};
\node at (-1, 1) {$\downarrow$};
\node at (-1.5, -0.8) {\text{cars}};
\node at (5, 4.8) {$i \rightarrow$ spots};
%row labels
\foreach \x in {1,...,3}
  \node at (1+2*\x, 3) {$\x$};
%column labels
\foreach \y in {1,...,3}
  \node at (1, 3-2*\y) {$\y$};
%tree edges
\draw [thick] (3,-3)--(7,-1);
%leaves
\tdot{3}{-3}{blue}
\tdot{5}{1}{blue}
\tdot{7}{-1}{blue}
\end{scope}

%third subgraph
\begin{scope}[shift={(42,0)}]
%inversion graph
\draw [step=2] (2,2) grid (8,-4);
\node at (-1, 2.6) {$\pi_i$};
\node at (-1, 1) {$\downarrow$};
\node at (-1.5, -0.8) {\text{cars}};
\node at (5, 4.8) {$i \rightarrow$ spots};
%row labels
\foreach \x in {1,...,3}
  \node at (1+2*\x, 3) {$\x$};
%column labels
\foreach \y in {1,...,3}
  \node at (1, 3-2*\y) {$\y$};
%tree edges
\draw [thick] (3,-3)--(5,1);
\draw [thick] (3,-3)--(7,-1);
%leaves
\tdot{3}{-3}{blue}
\tdot{5}{1}{blue}
\tdot{7}{-1}{blue}
\end{scope}
\end{tikzpicture}

\caption{The four $1$-subgraphs of the inversion graph $G_{312}$, which are all valid.\label{fig:all_subgraphs}}

\end{figure}

\end{example}

\begin{example}\label{ex:invalid}
Consider the permutation $\pi = 321$, and the $1$-subgraph $S$ with edges $(1,2)$ and $(1,3)$ (see Figure~\ref{fig:invalid_example} below). The corresponding parking function is $p = \SubtoPF(S) = (1,1,1)$. But we saw in Example~\ref{ex:valid} above that $\OMVP{3}{p} = 312 \neq 321$. As such, the $1$-subgraph $S$ is invalid.

\begin{figure}[ht]
\centering
\begin{tikzpicture}[scale=0.28]
%inversion graph
\draw [step=2] (2,2) grid (8,-4);
\node at (-1, 2.6) {$\pi_i$};
\node at (-1, 1) {$\downarrow$};
\node at (-1.5, -0.8) {\text{cars}};
\node at (4.8, 4.8) {$i \rightarrow$ spots};
%row labels
\foreach \x in {1,...,3}
  \node at (1+2*\x, 3) {$\x$};
%column labels
\foreach \y in {1,...,3}
  \node at (1, 3-2*\y) {$\y$};
%tree edges
\draw [thick, out=15, in=-105] (3,-3) to (7,1);
\draw [thick] (3,-3)--(5,-1);
%leaves
\tdot{3}{-3}{blue}
\tdot{5}{-1}{blue}
\tdot{7}{1}{blue}
\end{tikzpicture}
\caption{An example of an invalid $1$-subgraph of $G_{321}$.\label{fig:invalid_example}}
\end{figure}
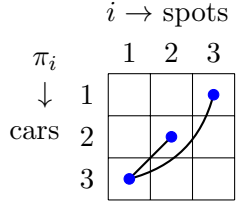

\end{example}

Another useful feature of the subgraph representation introduced in this section is that it allows certain statistics of parking functions to be easily read from the corresponding subgraph. Given a parking function $p$, the \emph{displacement} of car $i$ in $p$ as the number of spots car $i$ ends up from its original preference, i.e.\ $\vert p_i - \pi^{-1}_i \vert$. The displacement of $p$, denoted $\mathrm{disp}_{\mathrm{MVP}}(p)$, is simply the sum of the displacements of all cars in $p$. We have the following.

\begin{proposition}\label{pro:disp}
Let $p \in \MVP{n}$ be a parking function, and $S := \PFtoSub(p)$ the corresponding $1$-subgraph. Then we have $\mathrm{disp}_{\mathrm{MVP}}(p)= \sum\limits_{(j,i) \in S} (i-j)$.
\end{proposition}

\begin{proof}
Given an MVP parking function $p$, let $S := \PFtoSub(p)$ be the corresponding $1$-subgraph. By construction, we have an edge $(j,i)$ (with $j < i$) if the car $\pi_i$ that ends up in spot $i$ had initially preferred spot $j$. This means exactly that the displacement of car $\pi_i$ is given by $(i-j)$. The result follows immediately from this observation.
\end{proof}

\subsection{Improved bounds on the fibre sizes}\label{subsec:bounds}

In this part, we improve the upper bound from Corollary~\ref{cor:Harris_bound}, and also give a lower bound for the fibre size. We call a $1$-subgraph $S$ \emph{$\overrightarrow{P_2}$-free} if there is no triple $i < j < k$ such that $(i,j)$ and $(j,k)$ are both edges in $S$.

\begin{proposition}\label{pro:upper_bound}
Let $\pi \in S_n$ be a permutation, and $S \in \Sub{\pi}$ a $1$-subgraph of $G_{\pi}$. If $S$ is valid, then $S$ is $\overrightarrow{P_2}$-free. In particular, we have $\left\vert \FibMVP{n}{\pi} \right\vert \leq \vert \{ S \in \Sub{\pi}; \, S \text{ is } \overrightarrow{P_2}\text{-free} \} \vert$. 
\end{proposition}

\begin{proof}
We proceed by contraposition. Suppose that $S \in \Sub{\pi}$ is not $\overrightarrow{P_2}$-free. Define $p = \SubtoPF(S)$ to be the corresponding parking function, and $\pi' : = \OMVP{n}{p}$ to be its MVP outcome. We wish to show that $\pi' \neq \pi$. By definition, there is a triple $i<j<k$ such that $(i,j)$ and $(j,k)$ are both edges in the $S$ (see Figure~\ref{fig:example_triple}). Since edges are inversions in $\pi$, this implies that $\pi_i>\pi_j>\pi_k$, i.e.\ that car $\pi_k$ arrives earlier than car $\pi_j$, and car $\pi_j$ arrives earlier than car $\pi_i$. Without loss of generality, assume that $i$ is the left-most vertex of such a ``chain'', i.e.\ that the vertex $i$ has no incident left edge in $S$. This implies by Definition~\ref{def:SubtoPF} that $p_{\pi_i}=i$. Also by this definition, we get $p_{\pi_j}=i$ (because $(i,j)$ is an edge of $S$), and $p_{\pi_k} = j$. 

In particular we have $p_{\pi_i}=p_{\pi_j}$. Since $\pi_i>\pi_j$, this implies that car $\pi_j$ will be bumped out of its spot at the latest by car $\pi_i$ (it may be bumped first by another car also preferring the same spot). However,  when this bumping occurs, spot $j$ was necessarily already occupied, either by car $\pi_k$ (since $p_{\pi_k}=j$), or by a car which arrived later and bumped car $\pi_k$. Hence, when car $\pi_j$ is bumped out of its preferred spot, it could not finally park in spot $j$. This implies that the car $\pi'_j$, which ends up occupying spot $j$ in the MVP parking process for $p$, cannot be $\pi_j$, i.e.\ $\pi'_j \neq \pi_j$, and thus $\pi' \neq \pi$ as desired.

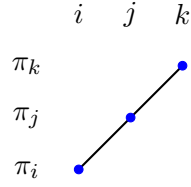
\begin{figure}[ht]
\centering
\begin{tikzpicture}[scale=0.23]
%inversion graph
\node at (3, 3) {$i$};
\node at (6, 3) {$j$};
\node at (9, 3) {$k$};
\node at (0, 0) {$\pi_k$};
\node at (0, -3) {$\pi_j$};
\node at (0, -6) {$\pi_i$};
%tree edges
\draw [thick] (3,-6)--(6,-3);
\draw [thick] (6,-3)--(9,0);
%leaves
\tdot{3}{-6}{blue}
\tdot{6}{-3}{blue}
\tdot{9}{0}{blue}
\end{tikzpicture}
\caption{A triple $i<j<k$ in $1$-subgraph $S$\label{fig:example_triple}}
\end{figure}

\end{proof}

We now give a lower bound on the fibre sizes. We say that a $1$-subgraph $S \in \Sub{\pi}$ is \emph{horizontally separated} (HS) if for any pair of edges $(j,i)$ and $(j',i')$ of $S$, we either have $i < j'$ or $i' < j$. In words, there is no pair of edges in $S$ which ``overlap horizontally'' in the graphical representation, end-points included. For example, in Figure~\ref{fig:example_HS}, Cases~(A) and (B) are HS with the condition $i<j'$ and $i'<j$ respectively, while Cases~(C) and (D) are not HS since $i'>j$ and $i>j'$ respectively.

\begin{figure}[ht]
  \centering
    \begin{tikzpicture}[scale=0.2]
    %one type: j, i, j', i'
    \draw [thick] (2,-11)--(5,-8);
    \draw [thick] (8,-5)--(11,-2);
    %border lines
    \draw [thick, dashed, green] (5,-2)--(5,-11);
    \draw [thick, dashed, green] (8,-2)--(8,-11);
    % dots
    \tdot{2}{-11}{blue}
    \tdot{5}{-8}{blue}
    \tdot{8}{-5}{blue}
    \tdot{11}{-2}{blue}
    % index
    \node at (2,1) {$j$};
    \node at (5,1) {$i$};
    \node at (8,1) {$j'$};
    \node at (11,1) {$i'$};
    
    \node at (-1, -2) {$\pi_{i'}$};
	\node at (-1, -5) {$\pi_{j'}$};
	\node at (-1, -8) {$\pi_i$};
	\node at (-1, -11) {$\pi_j$};
    % comment
    \node at (6.5,-13.5) {(A)};

	%another type: j', i', j, i
	\begin{scope}[shift={(20,0)}]
	\draw [thick] (2,-8)--(5,-5);
    \draw [thick] (8,-11)--(11,-2);
    %border lines
    \draw [thick, dashed, green] (5,-2)--(5,-11);
    \draw [thick, dashed, green] (8,-2)--(8,-11);
    % dots
    \tdot{2}{-8}{blue}
    \tdot{5}{-5}{blue}
    \tdot{8}{-11}{blue}
    \tdot{11}{-2}{blue}
    % index
    \node at (2,1) {$j'$};
    \node at (5,1) {$i'$};
    \node at (8,1) {$j$};
    \node at (11,1) {$i$};
    
    \node at (-1, -2) {$\pi_i$};
	\node at (-1, -5) {$\pi_{i'}$};
	\node at (-1, -8) {$\pi_{j'}$};
	\node at (-1, -11) {$\pi_j$};
    % comment
    \node at (6.5,-13.5) {(B)};
	\end{scope}
	
	%not allowed type1
	\begin{scope}[shift={(40,0)}]
	\draw [thick] (2,-8)--(11,-5);
    \draw [thick] (5,-11)--(8,-2);
    %border lines
    \draw [thick, dashed, green] (5,-2)--(5,-11);
    \draw [thick, dashed, green] (11,-2)--(11,-11);
    % dots
    \tdot{2}{-8}{blue}
    \tdot{5}{-11}{blue}
    \tdot{8}{-2}{blue}
    \tdot{11}{-5}{blue}
    % index
    \node at (2,1) {$j'$};
    \node at (5,1) {$j$};
    \node at (8,1) {$i$};
    \node at (11,1) {$i'$};
    
    \node at (-1, -2) {$\pi_{i}$};
	\node at (-1, -5) {$\pi_{i'}$};
	\node at (-1, -8) {$\pi_{j'}$};
	\node at (-1, -11) {$\pi_j$};
    % comment
    \node at (6.5,-13.5) {(C)};
	\end{scope}

	%not allowed type2
	\begin{scope}[shift={(60,0)}]
	\draw [thick] (2,-5)--(11,-2);
    \draw [thick] (5,-11)--(8,-8);
    %border lines
    \draw [thick, dashed, green] (5,-2)--(5,-11);
    \draw [thick, dashed, green] (11,-2)--(11,-11);
    % dots
    \tdot{2}{-5}{blue}
    \tdot{5}{-11}{blue}
    \tdot{8}{-8}{blue}
    \tdot{11}{-2}{blue}
    % index
    \node at (2,1) {$j$};
    \node at (5,1) {$j'$};
    \node at (8,1) {$i'$};
    \node at (11,1) {$i$};
    
    \node at (-1, -2) {$\pi_i$};
	\node at (-1, -5) {$\pi_j$};
	\node at (-1, -8) {$\pi_{i'}$};
	\node at (-1, -11) {$\pi_{j'}$};
    % comment
    \node at (6.5,-13.5) {(D)};
	\end{scope}
    
    \end{tikzpicture}
    \caption{Examples of HS and non-HS graphs: (A) and (B) are HS; (C) and (D) are not HS.\label{fig:example_HS}}
\end{figure}
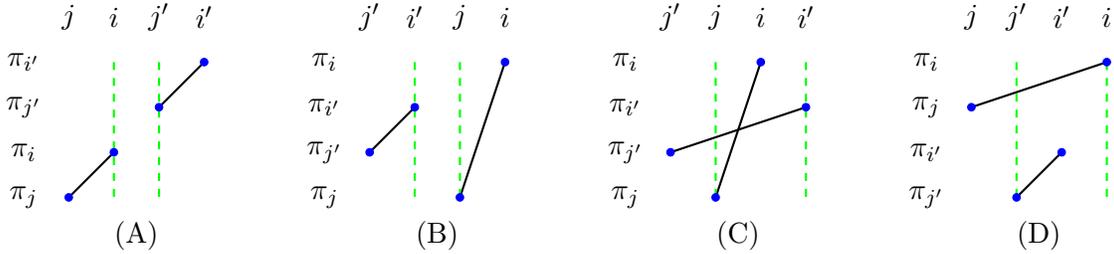

\begin{proposition}\label{pro:lower_bound}
Let $\pi \in S_n$ be a permutation, and $S \in \Sub{\pi}$ a $1$-subgraph of $G_{\pi}$. If $S$ is HS, then $S$ is valid. In particular, we have $\left\vert \FibMVP{n}{\pi} \right\vert \geq \vert \{ S \in \Sub{\pi}; \, S \text{ is HS} \} \vert$. 
\end{proposition}

\begin{proof}
Let $S \in \Sub{\pi}$ be HS, $p := \SubtoPF(S)$ the corresponding parking function, and $\pi' = \OMVP{n}{p}$ its outcome. We wish to show that $\pi' = \pi$. We proceed by induction on the number $k \geq 0$ of edges of $S$. For $k = 0$, $S$ is the empty subgraph with no edges. By construction, this means that $p_{\pi_i} = i$ for all $i \in [n]$. In particular, all cars have distinct preferences. As such, there are no bumpings/collisions, and every car ends up in its preferred spot in the outcome $\pi'$, i.e.\ $\pi'_i = \pi_i$ for all $i \in [n]$, as desired.

Suppose now that the result is proved for any HS subgraph $S'$ with at most $k-1$ edges, and that $S$ is a HS subgraph with $k$ edges. Let $(j,i)$, with $j < i$ and $\pi_j > \pi_i$, denote the right-most edge in $S$, i.e.\ (given the HS property) all other edges $(j', i')$ with $j' < i'$ must satisfy $i' < j$. Let $S'$ denote the subgraph $S$ with the edge $(j,i)$ removed. By construction, $S'$ is HS. We define $p' := \SubtoPF(S')$ to be the corresponding parking function, whose outcome is $\pi$ by the induction hypothesis ($S'$ is valid).

By construction, we have $p_{\pi_i} = j$, $p'_{\pi_i} = i$, and $p_k = p'_k$ for all other values of $k$. Moreover, by the HS assumption, cars $\pi_i$ and $\pi_j$ are the only cars to prefer spot $j$ in $p$. Now consider the MVP parking process for $p$. Up to the arrival of car $\pi_j$, the only difference compared with the process for $p'$ lies in the fact that car $\pi_i$ occupies spot $j$ instead of spot $i$. Otherwise the process is identical, and in particular spot $i$ is unoccupied at that point. When car $\pi_j$ arrives, it parks in spot $j$, and bumps $\pi_i$. Similarly, after the arrival of $\pi_j$, the processes in $p$ and $p'$ are the same, except perhaps for what happens to car $\pi_i$. It therefore suffices to show that $\pi_i$ will finally end up in spot $i$.

By the preceding remark, spot $i$ is available when car $\pi_j$ arrives, so car $\pi_i$ will first park in some available spot $i' \leq i$. But by the HS assumption of $S$, all the vertices between $j$ and $i$ are isolated. This means that $p_{\pi_{i'}} = i'$ for all $i' \in [j+1, i-1]$. In particular, all such spots $i'$ are either already occupied when $\pi_i$ is first bumped (if $\pi_{i'} < \pi_j$), or the arrival of car $\pi_{i'}$ will subsequently bump car $\pi_i$ out of the spot $i'$ (if $\pi_{i'} > \pi_j$). This implies that car $\pi_i$ cannot end up in such a spot $i'$, and therefore it can only end up in $i$, as desired.
\end{proof}

Note that any subgraph consisting of a single arc is horizontally separated, as is the empty subgraph. As such, Proposition~\ref{pro:lower_bound} implies in particular that  $\left\vert \FibMVP{n}{\pi} \right\vert \geq 1 + \left\vert \Inv{\pi} \right\vert$.
One might ask how tight the bounds of Propositions~\ref{pro:upper_bound} and \ref{pro:lower_bound} are. This is discussed in Section~\ref{sec:future} at the end of the paper.

%%%%%%%%%%%%%%%%%%%%%%%% SECTION 3 %%%%%%%%%%%%%%%%%%%%%

\section{Motzkin parking functions}\label{sec:MotzPF}

\subsection{Motzkin parking functions and Motzkin paths}\label{subsec:Motz_general}

We consider lattice paths starting from $(0,0)$ with steps $U = (1, 1)$ (upwards step), $D = (1, -1)$ (downwards step), and $H = (1, 0)$ (horizontal step). A \emph{Motzkin path} is a lattice path with these steps ending at some point $(n,0)$ which never goes below the X-axis (see Figure~\ref{fig:PF_Motz}). We denote $\Motz{n}$ the set of Motzkin paths ending at $(n,0)$ (i.e.\ with $n$ steps). Motzkin paths are enumerated by the ubiquitous \emph{Motzkin numbers} and are in bijection with a number of different combinatorial objects (see e.g.~\cite{BarMotz} or~\cite{StanleyEC}).

Given a parking function $p \in \PF{n}$, we define a lattice path with $n$ steps $\Phi(p) := \phi_1 \cdots \phi_n$ by:
\beq\label{eq:PF_Motz}
\forall j \in [n], \quad
\phi_j = 
\begin{cases}
U \quad \text{if } \quad \vert \{j;p_i=j\} \vert  \geq 2,\\
H \quad \text{if } \quad \vert \{j;p_i=j\} \vert = 1,\\
D \quad \text{if } \quad \vert \{j;p_i=j\} \vert = 0.
\end{cases}
\eeq

\begin{definition}
Let $p=(p_1,\cdots,p_n) \in \PF{n}$. We say that $p$ is a \emph{Motzkin parking function} if $\forall j \in [n]$, $\vert \{j;p_i=j\} \vert \leq 2$. We denote $\MotzPF{n}$ the set of Motzkin paking functions of length $n$.
\end{definition}

In words, a Motzkin parking function is a parking function in which each spot is preferred by at most two cars. The terminology of Motzkin parking function comes from the following result.

\begin{theorem}\label{thm:MotzPF_characterisation}
\label{thm:prefConfig_Motzkin}
Let $p$ be a parking preference. Then $p \in \MotzPF{n}$ if and only if $\Phi(p)$ is a Motzkin path. 
\end{theorem}

\begin{proof}
Let $p \in \MotzPF{n}$, meaning each spot will be preferred by at most $2$ cars in $p$. Denote $\phi := \Phi(p) = \phi_1 \cdots \phi_n$ the corresponding lattice path. We wish to show that $\phi$ is a Motzkin path. Fix some index $k \in [n]$ such that $\phi_k = U$. This means that spot $k$ is preferred by two cars, say $i$, $j$, with $i < j$ (i.e., $p_i=p_j=k$). Because one spot can only hold one car, this means that after bumping, car $j$ will park on spot $k$, and the other car $i$ will be bumped from the current spot $k$ to the first available spot. Denote $k'$ the spot where it eventually ends up once all cars have parked. 

We must have $k' > k$, and we claim that spot $k'$ is preferred by no cars. Seeking contradiction, suppose that $k'$ is preferred by a car $i'$, and let $j'$ denote the last car to arrive in the parking sequence before car $i$ parks in $k'$. On the one hand, if $i' < j'$, then it is impossible for $i$ to park in $k'$, since it would already be occupied at that point. On the other hand, if $i' > j'$, then car $i'$ would bump $i$ out of spot $k'$, so car $i$ cannot end up in spot $k'$. As such, $k'$ is preferred by no cars, i.e.\ $\phi_{k'} = D$.

We have therefore defined a map $k \mapsto k'$ which maps $U$ steps in $\Phi$ to $D$ steps, such that $k < k'$ for any $U$ step $k$. We claim that this map is bijective by exhibiting its inverse. If $k' \in [n]$ is such that $\phi_{k'} = 0$, we simply set $k$ to be the spot which was initially preferred by that car which ends up in spot $k'$. By construction, that car did not initially prefer $k'$ (no car did, since $\phi_{k'} = 0$), and therefore its initial preference must be to the left of $k'$, i.e. $k < k'$. Moreover, since that car ends up in a spot that is not its first preference, it must have been bumped out, meaning that its initial preference $k$ must be preferred by at least two cars, i.e. $\phi_k = 2$. This shows that the map $k \mapsto k'$ is indeed bijective, which implies that $\phi$ is a Motzkin path, as desired.

Conversely, suppose that $p \in \MVP{n}$ is such that $\phi = \Phi(p)$ is a Motzkin path. We wish to show that $p$ is a Motzkin parking function. Let $\mathcal{U} := \{k \in [n]; \, \phi_k = U\}$, resp.\ $\mathcal{D} := \{k \in [n]; \, \phi_k = D\}$, denote the set of $U$, resp.\ $D$, steps in $\Phi$. For any $k \in \mathcal{U}$, consider the set $C_k$ of cars that prefer spot $k$. We wish to show that $\vert C_k \vert = 2$ for all $k$. Let $i_k$ denote the car that ends up in spot $k$ after all cars have parked. By construction $i_k \in C_k$. As above, to each element $j \neq i_k$ in $C_k$, we can associate injectively a parking spot $k_j > k$ where car $j$ ends up, and this spot $k_j$ must be preferred by no cars, i.e.\ $k_j \in \mathcal{D}$. This therefore describes an injection $\bigcup\limits_{k \in \mathcal{U}} \left( C_k \setminus \{i_k\} \right) \hookrightarrow \mathcal{D}$. Finally, we get:
$$ \vert \mathcal{D} \vert = \vert \mathcal{U} \vert = \sum\limits_{k \in \mathcal{U}} 1 \leq \sum\limits_{k \in \mathcal{U}} \left( \vert C_k \vert - 1 \right) \leq \vert \mathcal{D} \vert,$$
where the equality $\vert \mathcal{D} \vert = \vert \mathcal{U} \vert$ stems from the fact that $\phi$ is a Motzkin path, the left inequality follows from the fact that $\vert C_k \vert \geq 2$ by construction of the map $\Phi$ (since $\phi_k = U$), and the right inequality follows from the injective map described above. For the left-most and right-most terms to be equal, we must therefore have that all terms in the sums are equal, i.e.\ that $\vert C_k \vert = 2$ for all $k \in \mathcal{U}$, as desired.
\end{proof}

\begin{example}\label{ex:PF_Motz}
Consider the Motzkin parking function $p=(2,2,1,4,3,6,4,6) \in \MotzPF{8}$. The corresponding Motzkin path is $\Phi(p)=HUHUDUDD$. Indeed, spots $1$ and $3$ are preferred by one car, spots $2$, $4$ and $6$ by two cars, and spots $5$, $7$ and $8$ by no cars. We can check that $\Phi(p)$, illustrated on Figure~\ref{fig:PF_Motz}, is indeed a Motzkin path.

\begin{figure}[ht]
\centering
\begin{tikzpicture}[scale=0.4]
%inversion graph
\draw [step=2] (2,2) grid (18,-2);
%Motzkin path
\draw [line width=2pt,red] (2,-2)--(4,-2);
\draw [line width=2pt,red] (4,-2)--(6,0);
\draw [line width=2pt,red] (6,0)--(8,0);
\draw [line width=2pt,red] (8,0)--(10,2);
\draw [line width=2pt,red] (10,2)--(12,0);
\draw [line width=2pt,red] (12,0)--(14,2);
\draw [line width=2pt,red] (14,2)--(16,0);
\draw [line width=2pt,red] (16,0)--(18,-2);

%leaves
\tdot{2}{-2}{red}
\tdot{18}{-2}{red}
\end{tikzpicture}
\caption{The Motzkin path $\Phi(p)=HUHUDUDD$ corresponding to the Motzkin parking function $p=(2,2,1,4,3,6,4,6)$.\label{fig:PF_Motz}}
\end{figure}
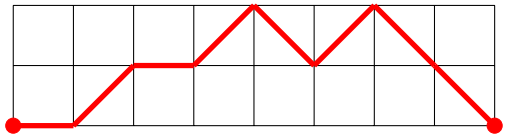

\end{example}

The definition of the map $\Phi$ in Equation~\eqref{eq:PF_Motz} only depends on the number of cars which prefer each spot, and not on the labels of the cars in question. It is therefore natural to define an equivalence relationship $\sim$ on $\MotzPF{n}$ by $p \sim p'$ if $p'$ is obtained by permuting the preferences in $p$. For example, the parking functions $(2, 1, 1, 4)$ and $(1, 4, 1, 2)$ are equivalent. 
We write $\faktor{\MotzPF{n}}{\sim}$ for the set of equivalence classes of Motzkin parking functions. The above observation implies that $\Phi$ is constant on the equivalence classes of $\sim$, so that with slight abuse of notation, we can consider $\Phi$ to be defined on the set $\faktor{\MotzPF{n}}{\sim}$. We then have the following.

\begin{theorem}\label{thm:bij_MotzPFequiv_Motz}
The map $\Phi: \faktor{\MotzPF{n}}{\sim}  \rightarrow \Motz{n}$ is a bijection.
\end{theorem}

\begin{proof}
As discussed above, the map is well defined since $\Phi$ is constant on the equivalence classes of $\MotzPF{n}$, and by Theorem~\ref{thm:prefConfig_Motzkin} its image is in $\Motz{n}$. 
To show that it is a bijection, we exhibit its inverse. Given a Motzkin path $\phi = \phi_1 \cdots \phi_n$, we construct a parking function using the following algorithm. 
\begin{enumerate}
\item Initialise $i = k = 1$.  
\item While $k \leq n$, do the following.
  \begin{enumerate}
  \item If $\phi_k = U$, set $p_i = p_{i+1} = k$ and $i = i+2$.
  \item If $\phi_k = H$, set $p_i = k$ and $i = i+1$.
  \item If $\phi_k = D$, do nothing (leave $i$ unchanged).
  \end{enumerate}
In all cases set $k = k+1$.
\item Output $p = (p_1, \cdots, p_n)$.
\end{enumerate}
Since $\phi$ has the same number of $U$ and $D$ steps, it is straightforward to check that this algorithm does indeed output a parking preference $p$. Moreover, by construction of the map $\Phi$, we have $\Phi(p) = \phi$, i.e.\ this construction defines an inverse for the map $\Phi$. It therefore remains to show that $p$ is indeed a parking function. Note that $p$ is non-decreasing by construction, i.e.\ $p_1 \leq \cdots \leq p_n$. As such, it is sufficient to show that for all $i \in [n]$, we have $p_i \leq i$ (see e.g.~\cite[Section~1.1]{Yan}).

For this, we first describe the construction of $p$ from $\phi$ in a slightly alternate form. We start from the empty sequence $S$, and initialise $s=1$ for the ``current spot''. Looking at the steps of the Motzkin path $\phi$ from left to right, we put the value $s$ into the sequence $S$ twice if we encounter a $U$ step, once if we encounter a $H$ step, and zero times if we encounter a $D$ step, then increase $s$ by one, and repeat (moving to the next step of the Motzkin path).

Because in any prefix  $\phi_1 \cdots \phi_k$ of $\phi$ the number of $U$ steps must be greater than or equal to the number of $D$ steps, this equivalent formulation implies that after $k$ iterations of the above algorithm, the length of the sequence $S$ must be greater than or equal to $k$. These iterations correspond exactly to placing the values $1, 2, \cdots, k$ into the sequence $S$ (since the current spot $s$ increases by one at each iteration). In other words, we have $\vert \{ i \in [n]; \, p_i \leq k \} \vert \geq k$ for all $k \in [n]$. We claim that this implies that for any $i \in [n]$, we have $p_i \leq i$. Otherwise, suppose that $p_i > i$ for some $i$. Since $p$ is non-decreasing, this means that $\{ j \in [n]; p_j \leq i \} \subseteq [i-1]$. But by the above, the left-hand set should have cardinality at least $i$, which yields the desired contradiction. Therefore we do indeed have $p_i \leq i$ for all $i \in [n]$, and so $p$ is a parking function, as desired. This completes the proof.
\end{proof}

Theorem~\ref{thm:bij_MotzPFequiv_Motz} can be viewed as a generalisation of the bijection between Motzkin paths and parking functions whose MVP outcome is the \emph{decreasing permutation} $\dec{n} := n (n-1) \cdots 1$, which was established by Harris \emph{et al.}~\cite{HarrisMVP}. We re-state that result in our context of Motzkin parking functions and their equivalence classes.

\begin{theorem}[{\cite[Theorem~4.2]{HarrisMVP}}]
\label{thm:bij_prefConfig_Motzkin}
For any $p \in \MotzPF{n}$,  there exists a unique $p' \in \MotzPF{n}$ such that $p \sim p'$, and $\OOMVP{n}(p') = \dec{n} $. In particular, $\Phi$ induces a bijection from the decreasing fibre $\FibMVP{n}{\dec{n}}$ to the set $\Motz{n}$ of Motzkin paths of length $n$.
\end{theorem}

\subsection{Non-crossing arc diagrams}\label{subsec:Motz_non_crossing_diagrams}

Theorem~\ref{thm:bij_prefConfig_Motzkin} implies in particular that the decreasing fibres $\FibMVP{n}{\dec{n}}$ are enumerated by the Motzkin numbers. In this part we give a new bijective explanation of this fact by using our subgraph representation from Section~\ref{sec:general_case}. 

Note that the inversion graph of the decreasing permutation $\dec{n}$ is just the complete graph $K_n$ on $n$ vertices, since all pairs are inversions in $\dec{n}$. We can therefore in a sense forget the geometry of the inversion graph, and it will be convenient to think of subgraphs of $G_{\dec{n}}$ as \emph{arc diagrams}. An arc diagram is simply a subset of edges of $G_{\dec{n}}$, i.e.\ a set of (some) pairs $(j,i) \in [n]^2$ with $j < i$. In this context, $\Sub{\dec{n}}$ is the set of arc diagrams of $[n]$ such that for any $i \in [n]$ there is at most one arc $(j,i)$ with $j < i$.

We say that an arc diagram $\Delta \in \Sub{\dec{n}}$ is a \emph{non-crossing matching} if it satisfies the two following conditions.
\begin{enumerate}
\item \textbf{Matching condition}: for every vertex $i \in [n]$, there is at most one arc incident to $i$ in $\Delta$.
\item \textbf{Non-crossing condition}: no two arcs of $\Delta$ ``cross'', that is there are no four vertices $i < j < k < \ell$ such that $(i,k)$ and $(j, \ell)$ are both arcs in $\Delta$.
\end{enumerate}
We denote $\NC_n$ the set of non-crossing arc diagrams on $[n]$. It is well-known that non-crossing arc diagrams are enumerated by Motzkin numbers. For a simple bijection between $\NC_n$ and $\Motz{n}$, simply map $\Delta \in \NC_n$ to $\phi = \phi_1 \cdots \phi_n \in \Motz{n}$ by setting $\phi_i = U$ if $i$ is incident to a right-arc $(i, j)$ in $\Delta$ with $i < j$, $\phi_i = D$ if $i$ is incident to a left-arc $(j, i)$ in $\Delta$ with $j < i$, and $\phi_i = H$ if $i$ is an isolated vertex  in $\Delta$. Our main result of this section is the following, which gives an alternate proof of the enumerative consequence of Theorem~\ref{thm:bij_prefConfig_Motzkin}.

\begin{theorem}\label{thm:nonCrossing_dec}
Let $n \geq 1$. The map $\SubtoPF$ is a bijection from the set of non-crossing matchings $\NC_n$ to the decreasing fibre $\FibMVP{n}{\dec{n}}$.
\end{theorem}

\begin{example}\label{ex:NonCrossing}

Consider the non-crossing matching $\Delta$ in Figure~\ref{fig:NonCrossing} below. We wish to compute the corresponding MVP parking function $p := \SubtoPF(\Delta)$. To get the parking preference $p_i$ of car $i$, we look at the vertex $n+1-i$ (equivalently, the $i$-th vertex from the right). If there is a left-arc incident to that vertex, we set $p_i$ to be the label of the left end-point $j$ of that arc. Otherwise we set $p_i = n+1-i$. We get the parking function $p = (11, 7, 8, 8, 7, 1, 3, 4, 3, 2, 1)$. One can check that we do indeed have $\OMVP{11}{p} = \dec{11}$, as desired.

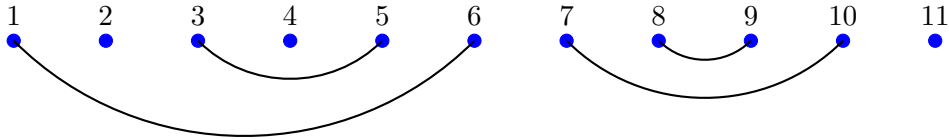
\begin{figure}[ht]
  \centering
    \begin{tikzpicture}[scale=0.35]
    \foreach \xx in {1,...,11} {
      \arcdot{3.5*\xx-3.5}{$\xx$}{blue}
    }
    \arcdraw{0}{17.5}
    \arcdraw{7}{14}
    \arcdraw{21}{31.5}
    \arcdraw{24.5}{28}
    \end{tikzpicture}
  \caption{An example of a non-crossing matching $\Delta$ on $11$ vertices. The corresponding parking function is $p := \SubtoPF(\Delta) = (11, 7, 8, 8, 7, 1, 3, 4, 3, 2, 1)$.\label{fig:NonCrossing}}
\end{figure}

\end{example}

\begin{proof}[Proof of Theorem~\ref{thm:nonCrossing_dec}]
Note that we already know by Theorem~\ref{thm:MVPPF_subgraphs} that the map is injective, so it suffices to show that $\SubtoPF \left(\NC_n \right) = \FibMVP{n}{\dec{n}}$, or equivalently, that $\Valid{\dec{n}} = \NC_n$. We first show that inclusion $\Valid{\dec{n}} \subseteq \NC_n$.

Let $\Delta \in \Valid{\dec{n}}$ be a valid $1$-arc diagram, and denote $p = (p_1, \cdots, p_n) := \SubtoPF(\Delta)$ the corresponding MVP parking function, satisfying $\OMVP{n}{p} = \dec{n}$. To simplify notation, for $i \in [n]$ we denote $\bar{i} := n+1-i$ the car that ends up in spot $i$.
We want to show that $\Delta$ is a non-crossing matching.
For this, we need to prove two conditions: the matching condition, and the non-crossing condition.

\textbf{Matching condition.} Seeking contradiction, suppose there is a vertex which is incident to two arcs. There are three cases to consider here.

Case (A): some vertex $k$ is incident to two left-arcs, That is, there exist $i < j < k$ such that $(i,k)$ and $(j,k)$ are arcs of $\Delta$ (see Figure~\ref{fig:two_edges_To_1}). This directly contradicts the condition that $\Delta$ is a $1$-arc diagram (in other words that only one car ends up in spot $k$), so there is nothing to show here.

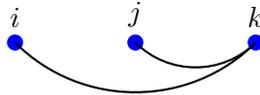
\begin{figure}[ht]
  \centering
    \begin{tikzpicture}[scale=0.4]
    \arcdot{0}{$i$}{blue}
    \arcdot{4}{$j$}{blue}
    \arcdot{8}{$k$}{blue}
    \arcdraw{4}{8};
    \arcdraw{0}{8}
    \end{tikzpicture}
  \caption{Case where $k$ is incident to two left-arcs.\label{fig:two_edges_To_1}}
\end{figure}

Case (B): some vertex $j$ is incident to both a left-arc and a right-arc. That is, there exist $i < j < k$ such that $(i,j)$ and $(j,k)$ are arcs of $\Delta$ (see Figure~\ref{fig:two_edges_consec}). In this case, $\Delta$ is not $\overrightarrow{P_2}$-free, so by Proposition~\ref{pro:upper_bound} $\Delta$ cannot be valid, which is a contradiction.

\begin{figure}[ht]
  \centering
    \begin{tikzpicture}[scale=0.4]
    \arcdot{0}{$i$}{blue}
    \arcdot{4}{$j$}{blue}
    \arcdot{8}{$k$}{blue}
    \arcdraw{0}{4}
    \arcdraw{4}{8}
    \end{tikzpicture}
  \caption{Case where $j$ is incident to both a left-arc and a right-arc.\label{fig:two_edges_consec}}
\end{figure}
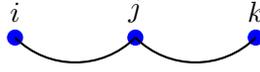

Case (C): some vertex $i$ is incident to two right-arcs. That is, there exist $i < j < k$ such that $(i,j)$ and $(i,k)$ are arcs of $\Delta$ (see Figure~\ref{fig:two_edges_To_3}).

\begin{figure}[ht]
  \centering
    \begin{tikzpicture}[scale=0.4]
    \arcdot{0}{$i$}{blue}
    \arcdot{4}{$j$}{blue}
    \arcdot{8}{$k$}{blue}
    \arcdraw{0}{4}
    \arcdraw{0}{8}
    \end{tikzpicture}
  \caption{Case where $i$ is incident to two right-arcs.\label{fig:two_edges_To_3}}
\end{figure}
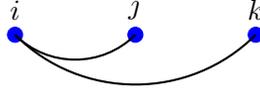

By Case~(B) we have that $i$ cannot be incident to a left-arc in $\Delta$. Therefore, in the parking function, we have $p_{\bar{i}}=p_{\bar{j}}=p_{\bar{k}}=i$. Without loss of generality, we may assume that $j$ and $k$ are the first two vertices incident to $i$, i.e. that there is no arc $(i, \ell)$ for $\ell < k$ and $\ell \neq j$. This is equivalent to cars $\bar{k}$, $\bar{j}$ and $\bar{i}$ being the last three cars to prefer spot $i$.

Because $\bar{k} < \bar{j} < \bar{i}$, car $\bar{k}$ arrives and parks on spot $i$ first, and then car $\bar{j}$ parks on spot $i$ again, bumping car $\bar{k}$. Since car $\bar{i}$ will subsequently bump car $\bar{j}$ from spot $i$ (no other car prefers spot $i$ in between by the above assumption), and car $\bar{j}$ needs to end up parking in spot $j$, this implies that spot $j$ must be free when car $\bar{i}$ arrives. In particular, this implies that after being bumped from spot $i$ by car $\bar{j}$, car $\bar{k}$ must park in some spot $\ell$ with $i < \ell < j$.

But car $\bar{k}$ should end up in spot $k$. For this to occur, it must therefore be bumped by another car, say $\bar{k'}$. According to the above, such a car must arrive after $\bar{i}$, i.e. $k' < i$, since $j$ must still be free when $\bar{i}$ arrives, but occupied when $\bar{k'}$ arrives so that $\bar{k}$ ends up in $k$, ``skipping'' over spot $j$. But by construction, the preferred spot of any car $\bar{k'}$ with $k' < i$ must be less than $i$, which contradicts the fact that car $\bar{k'}$ would bump $\bar{k}$ from a spot $\ell > i$. This concludes the proof of the matching condition.

\textbf{Non-crossing condition}. We now show that if $\Delta$ is a valid $1$-arc diagram, then $\Delta$ is non-crossing. Seeking contradiction, suppose that $\Delta$ has an edge-crossing. That is, there exist $i < j < k < \ell$ such that the edges $(i,k)$ and $(j,\ell)$ are both in the diagram $\Delta$ (see Figure~\ref{fig:edge_Crossing}).

\begin{figure}[ht]
  \centering
    \begin{tikzpicture}[scale=0.4]
    \arcdot{0}{$i$}{blue}
    \arcdot{4}{$j$}{blue}
    \arcdot{8}{$k$}{blue}
    \arcdot{12}{$\ell$}{blue}
    \arcdraw{0}{8}
    \arcdraw{4}{12}
    \end{tikzpicture}
  \caption{Case where the diagram $\Delta$ contains an edge-crossing.\label{fig:edge_Crossing}}
\end{figure}
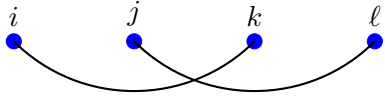

By construction, and applying Case~(B) from the matching condition above, we have that $p_{\bar{i}} = p_{\bar{k}} = i$, and $p_{\bar{j}} = p_{\bar{\ell}} = j$. Moreover, fixing $i$ and $k$, we can consider the right-most $j$ whose right-arc crosses the arc $(i,k)$, which also implies that $k$ is the left-most $k$ whose left-arc crosses the arc $(j, \ell)$. Therefore we can assume without loss of generality that for any $m$ with $j < m < k$, we have $j < p_{\bar{m}} < k$. In other words, the cars arriving between cars $\bar{k}$ and $\bar{j}$ occupy exactly the spots between $j$ and $k$. 

Because $i < j < k < \ell$ and $\bar{i} > \bar{j} > \bar{k} > \bar{\ell}$, car $\bar{\ell}$ first arrives and parks on spot $j$ (since $p_{\bar{\ell}}=j$). 
Then car $\bar{k}$ arrives and parks on spot $i$ according to the preference $p_{\bar{k}}=i$, which does not influence the car $\bar{\ell}$. 
However, after car $\bar{j}$ arrives, it prefers spot $j$ again, bumping the current car $\bar{\ell}$ which was previously occupying spot $j$. But by the assumption above all spaces between $j$ and $k$ have been occupied at this point (by cars arriving between $\bar{k}$ and $\bar{j}$). 
Therefore, after being bumped by car $\bar{j}$, car $\bar{\ell}$ parks in some spot $\ell'$ with $\ell' \geq k$. In particular, this implies that after car $\bar{j}$ has arrived, the spot $k$ must be occupied (either by $\bar{\ell}$ or by a previous car). 
As such, there is no way for car $\bar{k}$ to subsequently occupy spot $k$ after being bumped by car $\bar{i}$, which yields the desired contradiction.
This concludes the proof of the inclusion $\Valid{\dec{n}} \subseteq \NC_n$.

Before proving the converse, we first introduce some additional notation. We say that a non-crossing matching $\Delta \in \NC_n$ is \emph{prime} if it contains the arc $(1,n)$. A non-crossing matching $\Delta$ that is not prime can be uniquely decomposed into a disjoint union of \emph{prime factors}, where the prime factors are simply the non-crossing matchings induced on subsets $\{i, \cdots j\}$ such that $(i,j)$ is an arc of $\Delta$ and there is no arc $(i',j')$ of $\Delta$ with $i' < i$ and $j < j'$. Figure~\ref{fig:prime_decomp} shows an example of the prime decomposition of a non-crossing matching, with different colours corresponding to the three different prime factors.

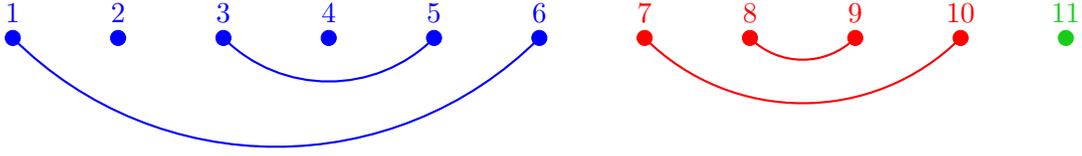
\begin{figure}[ht]
  \centering
    \begin{tikzpicture}[scale=0.4]
    \foreach \xx in {1,...,6} {
      \draw [blue, fill=blue] (3.5*\xx-3.5,0) circle [radius=0.25];
      \node [above, yshift=2pt] at (3.5*\xx-3.5,0) {\textcolor{blue}{$\xx$}};
    }
    \foreach \xx in {7,...,10} {
      \draw [red, fill=red] (3.5*\xx-3.5,0) circle [radius=0.25];
      \node [above, yshift=2pt] at (3.5*\xx-3.5,0) {\textcolor{red}{$\xx$}};
    }
    \draw [mygreen, fill=mygreen] (35,0) circle [radius=0.25];
    \node [above, yshift=2pt] at (35,0) {\textcolor{mygreen}{$11$}};
    \draw [thick, color=blue, out=-45, in=225] (0,0) to (17.5,0);
    \draw [thick, color=blue, out=-45, in=225] (7,0) to (14,0);
    \draw [thick, color=red, out=-45, in=225] (21,0) to (31.5,0);
    \draw [thick, color=red, out=-45, in=225] (24.5,0) to (28,0);
    \end{tikzpicture}
  \caption{An example of the prime decomposition of a non-crossing matching. The prime factors are the matchings induced on the intervals $[1,6]$ (in blue), $[7,10]$ (in red), and the isolated vertex $11$ (in green).\label{fig:prime_decomp}}
\end{figure}

We are now equipped to show the converse inclusion $\NC_n \subseteq \Valid{\dec{n}}$. That is, we show that for any non-crossing matching $\Delta \in \NC_n$, we have $p = \SubtoPF(\Delta) = (p_1,\cdots,p_n) \in \FibMVP{n}{\dec{n}}$ (i.e.\ $\Delta$ is valid). We proceed by induction on $n \geq 1$. For $n=1$, the only non-crossing matching consists of a single isolated vertex $1$, whose corresponding parking function is $p = (1)$. We then have $\OMVP{1}{p} = 1 = \dec{1}$, as desired.

For the induction step, fix $n > 1$ and suppose that the result holds for all $k < n$. Let $\Delta \in \NC_n$ be a non-crossing matching, and $p := \SubtoPF(\Delta) = (p_1,\cdots,p_n)$ the corresponding parking function. If $\Delta$ is not prime (does not contain the arc $(1,n)$), then the result follows immediately by applying the induction hypothesis to each prime factor of $\Delta$. Therefore it remains to consider the case where $(1,n)$ is an arc of $\Delta$. This means that $p_{\overline{1}} = p_{\overline{n}} = 1$ by construction, i.e.\ cars $1$ and $n$ both prefer spot $1$. Define $\Delta' \in \NC_{n-1}$ to be the non-crossing matching with vertex set $[n-1]$ obtained by deleting the vertex $n$ and the arc $(1,n)$ from $\Delta$, and let $p' := \SubtoPF(\Delta')$ be the corresponding parking function. Note that we have $p'_i = p_i$ for all $i \in [n-1]$. We apply the induction hypothesis to $\Delta'$, which means that $\OMVP{n-1}{p'} = \dec{n-1}$. In other words, following the MVP parking process for $p'$, car $\overline{i}$ ends up in spot $i$ for all $i \in [n-1]$. 

Now consider the MVP parking process for $p$. Car $1 = \overline{n}$ is the first to enter, and parks in its preferred spot $1$. Cars $\overline{n-1}, \cdots \overline{2}$ then enter in that order. Since none of these cars prefers spot $1$, car $1$ does not move during this process. Moreover, since $p_i = p'_i$ for all $i \in [n-1]$, the parking process of these cars for $p$ follows exactly that for $p'$. In particular, car $\overline{i}$ ends up in spot $i$ for all $i \in \{2, \cdots, n-1\}$. Finally, car $n = \overline{1}$ enters the car park. It parks in its preferred spot $1$, as desired. When it does so, it bumps car $1$. By the above, spots $2, \cdots, n-1$ have all been occupied at this point, so car $1 = \overline{n}$ parks in the (only) remaining spot $n$, as desired. This concludes the proof.
\end{proof}

\begin{remark}\label{rem:matching_not_general}
In the proof of the matching condition for a valid $1$-arc diagram $\Delta$, we can ask which of the three cases (A), (B), (C) are still forbidden for a general permutation $\pi \neq \dec{n}$. Case~(A) obviously still holds by definition of $1$-subgraphs. Case~(B) also still holds in the sense of Proposition~\ref{pro:upper_bound}. That is, there can be no three spots $i < j < k$ with cars $\pi_i > \pi_j > \pi_k$ such that $p_{\pi_k} = j$ and $p_{\pi_j} = p_{\pi_i} = i$.

However, Case~(C) no longer holds. Indeed, consider the parking preference $p = (1,1,1,2)$. In this case $\pi = \OMVP{4}{p} = 3421$. Figure~\ref{fig:321pattern} illustrates the corresponding subgraph $S \in \Sub{3421}$, which has edges $(1,3)$ and $(1,4)$ with $\pi_4 < \pi_3 < \pi_1$. The key difference is the presence of the car $k' = 4$ which arrives last and bumps car $1$ from spot $2$ to spot $4$. In the decreasing permutation case this could not occur. 

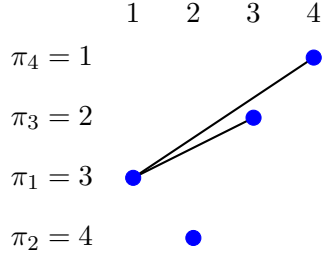
\begin{figure}[ht]
  \centering
    \begin{tikzpicture}[scale=0.4]
    \draw [thick] (2,-6)--(6,-4);
    \draw [thick] (2,-6)--(8,-2);
    % dots
    \tdot{2}{-6}{blue}
    \tdot{4}{-8}{blue}
    \tdot{6}{-4}{blue}
    \tdot{8}{-2}{blue}
    % index
    \node at (2,-0.5) {$1$};
    \node at (4,-0.5) {$2$};
    \node at (6,-0.5) {$3$};
    \node at (8,-0.5) {$4$};
    \node [left] at (1,-2) {$\pi_4 = 1$};
    \node [left] at (1,-4) {$\pi_3 = 2$};
    \node [left] at (1,-6) {$\pi_1 = 3$};
    \node [left] at (1,-8) {$\pi_2 = 4$};
    \end{tikzpicture}
  \caption{The subgraph corresponding to the MVP parking function $(1,1,1,2)$.\label{fig:321pattern}}
\end{figure}
\end{remark}

\begin{remark}\label{rem:non_crossing_not_general}
Similarly, we can see that the non-crossing condition does not hold for general valid $1$-subgraphs. Indeed, car $\bar{\ell}$ could first be bumped to a spot between $j$ and $k$ and be subsequently bumped into spot $\ell$ by a car arriving after $\bar{i}$. For example, consider the parking preference $p = (2,1,2,1,3)$, whose outcome is $\pi = \OMVP{5}{p} = 43521$. Figure~\ref{fig:4321pattern} illustrates the corresponding valid $1$-subgraph, which has edges $(1,4)$ and $(2,5)$ with $\pi_5 < \pi_4 < \pi_2 < \pi_1$ (corresponding to crossing edges in an arc diagram representation).

\begin{figure}[ht]
  \centering
    \begin{tikzpicture}[scale=0.4]
    \draw [thick] (2,-8)--(8,-4);
    \draw [thick] (4,-6)--(10,-2);
    % dots
    \tdot{2}{-8}{blue}
    \tdot{4}{-6}{blue}
    \tdot{6}{-10}{blue}
    \tdot{8}{-4}{blue}
    \tdot{10}{-2}{blue}
    % index
    \node at (2,-0.5) {$1$};
    \node at (4,-0.5) {$2$};
    \node at (6,-0.5) {$3$};
    \node at (8,-0.5) {$4$};
    \node at (10,-0.5) {$5$};
    \node [left] at (1,-2) {$\pi_5 = 1$};
    \node [left] at (1,-4) {$\pi_4 = 2$};
    \node [left] at (1,-6) {$\pi_2 = 3$};
    \node [left] at (1,-8) {$\pi_1 = 4$};
    \node [left] at (1,-10) {$\pi_3 = 5$};
    \end{tikzpicture}
  \caption{The subgraph corresponding to the MVP parking function $(2,1,2,1,3)$.\label{fig:4321pattern}}
\end{figure}
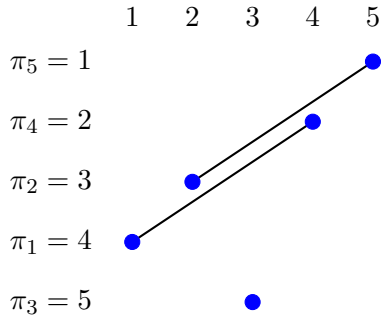
\end{remark}

\section{The complete bipartite case}\label{sec:m_n}

In this section we study the case where the inversion graph is the complete bipartite graph. Given $m,n \geq 1$, the \emph{complete bipartite permutation} is the permutation $\bipart{m}{n} := (n+1)(n+2)\cdots (n+m)12\cdots n$. In this case, the inversion graph $G_{\bipart{m}{n}}$ is the complete bipartite graph $K_{m,n}$, whose edges are all pairs $(i, m+j)$ for $i \in [m]$ and $j \in [n]$.

\begin{example}\label{ex:m_n}
Consider the permutation $\bipart{m}{n} = 4567123$ for $m=4$ and $n=3$. Figure~\ref{fig:m_n} shows one example of its $1$-subgraphs.

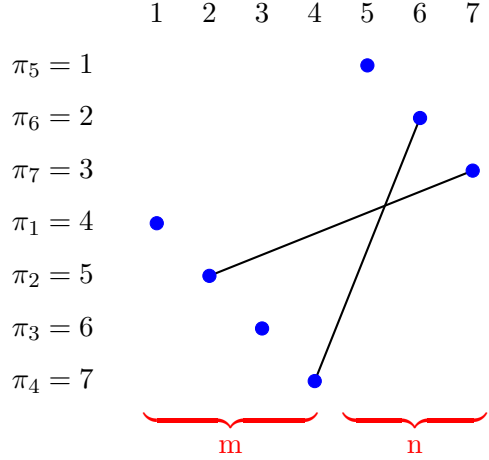
\begin{figure}[ht]
\centering
\begin{tikzpicture}[scale=0.35]
	%bipartite graph    
    %edges
    \draw [thick] (4,-9)--(14,-5);
    \draw [thick] (8,-13)--(12,-3);    
    % dots
    \tdot{2}{-7}{blue}
    \tdot{4}{-9}{blue}
    \tdot{6}{-11}{blue}
    \tdot{8}{-13}{blue}
    \tdot{10}{-1}{blue}
    \tdot{12}{-3}{blue}
    \tdot{14}{-5}{blue}
    % index
    \node at (2,1) {$1$};
    \node at (4,1) {$2$};
    \node at (6,1) {$3$};
    \node at (8,1) {$4$};
    \node at (10,1) {$5$};
    \node at (12,1) {$6$};
    \node at (14,1) {$7$};
    
    \node [left] at (0,-1) {$\pi_5 = 1$};
    \node [left] at (0,-3) {$\pi_6 = 2$};
    \node [left] at (0,-5) {$\pi_7 = 3$};
    \node [left] at (0,-7) {$\pi_1 = 4$};
    \node [left] at (0,-9) {$\pi_2 = 5$};
    \node [left] at (0,-11) {$\pi_3 = 6$};
    \node [left] at (0,-13) {$\pi_4 = 7$};
    
	% comment
		% m
    \node[rotate = 0, red] at (4.8, -14.5) {$\underbrace{\hspace{2.3cm}}$};
    \node[red] at (4.8,-15.5) {m};
    	%n
    \node[rotate = 0, red] at (11.8, -14.5) {$\underbrace{\hspace{1.9cm}}$};
    \node[red] at (11.8,-15.5) {n};
\end{tikzpicture}
\caption{One $1$-subgraph $S \in \Sub{\bipart{4}{3}}$ for the complete bipartite permutation $\bipart{4}{3} = 4567123$.\label{fig:m_n}}
\end{figure}
\end{example}

In this work, we focus on the case $n=2$, and obtain the following enumerative formula for the MVP outcome fibre.

\begin{theorem}\label{thm:enum_m_2}
For any $m \geq 0$, we have $|\mathcal{O}^{-1}_{\MVP{m+2}}(\bipart{m}{2})| = m+1 + \lfloor \frac{(m+1)^2}{2} \rfloor$.
\end{theorem}

We will prove the theorem by induction on $m$, by seeing how many ``additional'' valid $1$-subgraphs there are in $G_{\bipart{m+1}{2}}$ compared to those in $G_{\bipart{m}{2}}$.

\begin{lemma}\label{lem:addfirst_valid}
For $m \geq 1$, set $m' := m+1$. We define a map $\Sub{\bipart{m}{2}} \rightarrow \Sub{\bipart{m'}{2}}$, $S \mapsto S'$, as follows. For a $1$-subgraph $S \in \Sub{\bipart{m}{2}}$, we define a $1$-subgraph $S' \in \Sub{\bipart{m'}{2}}$ by increasing vertex labels in $S$ by one, and inserting a new isolated vertex labelled $1$ (see Figure~\ref{fig:insert_valid} below). Then the map $S \mapsto S'$ is a bijection from the set of valid $1$-subgraphs of $G_{\bipart{m}{2}}$ to the set of valid $1$-subgraphs of $G_{\bipart{m'}{2}}$ where the vertex labelled $1$ has no incident edge.
\end{lemma}

\begin{example}
Consider the $1$-subgraph $S \in \Sub{\bipart{m}{2}}$ with $m=3$ as in Figure~\ref{fig:insert_valid}, Graph~(A). Using the map $\SubtoPF(S)$, the corresponding parking function is $p=(2,1,1,2,3)$.
By running the MVP parking process we get $\OMVP{m+2}{p} = 34512 = \bipart{m}{2}$, so $S$ is valid. Then increasing vertex labels in $S$ by one and inserting a new isolated vertex labelled $1$, we get a new $1$-subgraph $S' \in \Sub{\bipart{m'}{2}}$ with $m'=m+1=4$ as in Figure~\ref{fig:insert_valid}, Graph~(B). For $S'$, the corresponding parking function is $p'=\SubtoPF(S')=(3,2,1,2,3,4)$. Then by running the MVP parking process we get $\OMVP{m'+2}{p'} = 345612 = \bipart{m'}{2} $, so $S'$ is also valid, as desired.

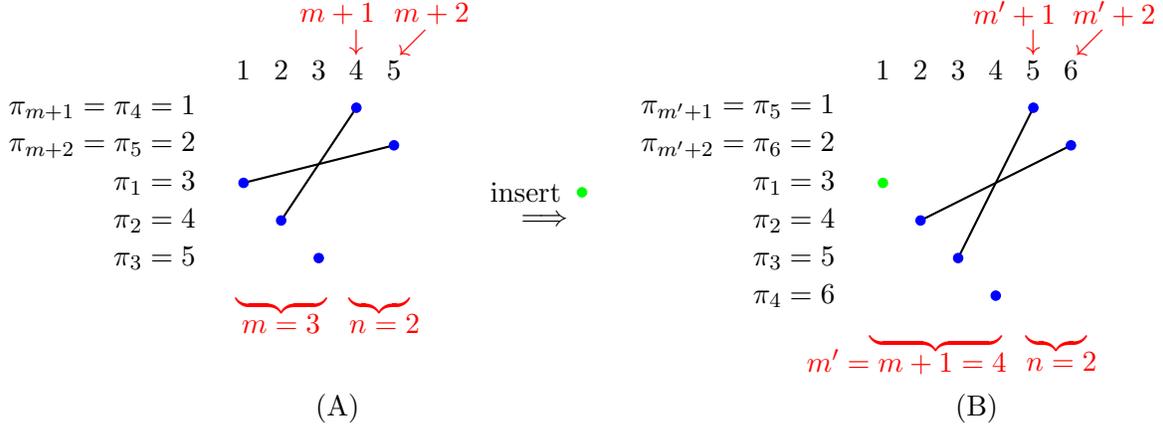
\begin{figure}[ht]
\centering
\begin{tikzpicture}[scale=0.25]
	%bipartite graph S
	% edges
	\draw[thick] (2,-5)--(10,-3);
	\draw[thick] (4,-7)--(8,-1);
    % dots
    \tdot{2}{-5}{blue}
    \tdot{4}{-7}{blue}
    \tdot{6}{-9}{blue}
    \tdot{8}{-1}{blue}
    \tdot{10}{-3}{blue}
    % index
    \node at (2,1) {$1$};
    \node at (4,1) {$2$};
    \node at (6,1) {$3$};
    \node[red] at (7,4) {$m+1$};
    \node[red] at (8,2.5) {$\downarrow$};
    \node at (8,1) {$4$};
    \node[red] at (12,4) {$m+2$};
    \node[red] at (11,2.5) {$\swarrow$};
    \node at (10,1) {$5$};
    
    \node [left] at (0,-1) {$\pi_{m+1}=\pi_{4} = 1$};
    \node [left] at (0,-3) {$\pi_{m+2}=\pi_{5} = 2$};
    \node [left] at (0,-5) {$\pi_{1} = 3$};
    \node [left] at (0,-7) {$\pi_{2} = 4$};
    \node [left] at (0,-9) {$\pi_{3} = 5$};
	% comment
		% m
    \node[rotate = 0, red] at (4, -11.5) {$\underbrace{\hspace{1.2cm}}$};
    \node[red] at (4,-12.5) {$m=3$};
    	%n
    \node[rotate = 0, red] at (9.2, -11.5) {$\underbrace{\hspace{0.8cm}}$};
    \node[red] at (9.5,-12.5) {$n=2$};
	% comment
    \node at (7,-17) {(A)};    
    \node at (18,-7) {$\Longrightarrow$};
    \node at (17,-5.5) {\text{insert}};
    \tdot{20}{-5.5}{green}
    
    %bipartite graph S'
	\begin{scope}[shift={(34,0)}]
	% edges
	\draw[thick] (4,-7)--(12,-3);
	\draw[thick] (6,-9)--(10,-1);
	% dots
    \tdot{2}{-5}{green}
    \tdot{4}{-7}{blue}
    \tdot{6}{-9}{blue}
    \tdot{8}{-11}{blue}
    \tdot{10}{-1}{blue}
    \tdot{12}{-3}{blue}
    % index
    \node at (2,1) {$1$};
    \node at (4,1) {$2$};
    \node at (6,1) {$3$};
    \node at (8,1) {$4$};
    \node[red] at (9,4) {$m'+1$};
    \node[red] at (10,2.5) {$\downarrow$};
    \node at (10,1) {$5$};
    \node[red] at (14.3,4) {$m'+2$};
    \node[red] at (13,2.5) {$\swarrow$};
    \node at (12,1) {$6$};
    
    \node [left] at (0,-1) {$\pi_{m'+1}=\pi_{5} = 1$};
    \node [left] at (0,-3) {$\pi_{m'+2}=\pi_{6} = 2$};
    \node [left] at (0,-5) {$\pi_{1} = 3$};
    \node [left] at (0,-7) {$\pi_{2} = 4$};
    \node [left] at (0,-9) {$\pi_{3} = 5$};
    \node [left] at (0,-11) {$\pi_{4} = 6$};
	% comment
		% m
    \node[rotate = 0, red] at (4.8, -13.5) {$\underbrace{\hspace{1.75cm}}$};
    \node[red, xshift=-1em] at (4.8,-14.5) {$m'=m+1=4$};
    	%n
    \node[rotate = 0, red] at (11.2, -13.5) {$\underbrace{\hspace{0.8cm}}$};
    \node[red] at (11.5,-14.5) {$n=2$};
    % comment
    \node at (7,-17) {(B)};
	\end{scope}
\end{tikzpicture}
\caption{Illustrating the construction from Lemma~\ref{lem:addfirst_valid}: (A) a valid $1$-subgraph $S \in \Sub{\bipart{m}{2}}$ with $m=3$; (B) the corresponding valid $1$-subgraph $S' \in \Sub{\bipart{m'}{2}}$ with $m' = m+1 = 4$, obtained by inserting an isolated vertex in column $1$, row $3$, and shifting other vertices in $S$.\label{fig:insert_valid}}
\end{figure}
\end{example}

\begin{proof}[Proof of Lemma~\ref{lem:addfirst_valid}]
We first show that if $S$ is valid, so is $S'$. Let $p := \SubtoPF(S)$, resp.\ $p' := \SubtoPF(S')$ denote the parking function corresponding to $S$, resp.\ $S'$.
We consider the MVP parking process for $p'$. By construction, since the vertex $1$ is isolated in $S'$, and $\bipart{m'}{2}_1 = 3$ is isolated in $S'$, only the car $3$ prefers the spot $1$. This means that car $3$ will end up in spot $1$. Moreover, the remaining $m+2$ cars follow the same parking process as in $p$, with their spots increased by one. Since $S$ is valid, i.e.\ $\OMVP{m+2}{p} = \bipart{m}{2}$, this implies that $\OMVP{m'+2}{p'} = \bipart{m'}{2}$, i.e.\ $S'$ is valid, as desired.

We now claim that if $S' \in \Sub{\bipart{m'}{2}}$ is a valid $1$-subgraph where the vertex $1$ is isolated, then deleting that vertex and decreasing labels of remaining vertices by one  yields a $1$-subgraph $S \in \Sub{\bipart{m}{2}}$ which is also valid. 
Since the operations of inserting and deleting an isolated vertex (with suitable re-labelling of others) are clearly inverses of each other, this suffices to complete the proof. 
Suppose therefore that $S' \in \Sub{\bipart{m'}{2}}$ is a valid $1$-subgraph where the vertex $1$ is isolated. Let $S$ be the $1$-subgraph of $G_{\bipart{m}{2}}$ obtained by deleting this vertex, and decreasing labels of other vertices by $1$. As above, denote by $p$ and $p'$ the associated parking functions.
By construction, we have $p'_3 = 1$, and this is the only car preferring spot $1$. Therefore, in $p$, the parking process follows that of $p'$ where we have removed car $3$ and spot $1$ (with suitable re-labelling of remaining cars and spots). As such, since $S'$ is valid, it follows that $S$ is also valid, as desired.
\end{proof}

Lemma~\ref{lem:addfirst_valid} implies that the number of valid $1$-subgraphs of $G_{\bipart{m+1}{2}}$ where the vertex $1$ is isolated is equal to the number of valid $1$-subgraphs of $G_{\bipart{m}{2}}$. We now enumerate valid $1$-subgraphs of $G_{\bipart{m+1}{2}}$ where vertex $1$ has a single neighbour.

\begin{lemma}\label{lem:twoedges_1_j_valid}
For any $j\in[2,m]$, there are two $1$-subgraphs in $G_{\bipart{m}{2}}$ with one edge incident to $j$ and another incident to $1$. These are the subgraphs $S_1$ with edges $(1,m+1)$ and $(j, m+2)$, and $S_2$ with edges $(1, m+2)$ and $(j, m+1)$ (see Figure~\ref{fig:twoedges_twopatterns}). Then for any $j \in [2, m]$, exactly one of these two $1$-subgraphs is valid.
\end{lemma}

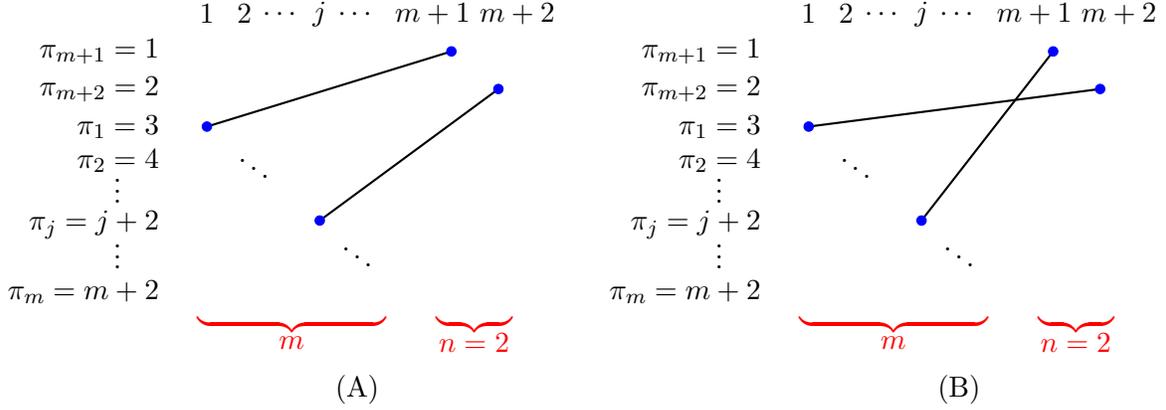
\begin{figure}[ht]
\centering
\begin{tikzpicture}[scale=0.25]
	%bipartite graph S 
	% edges
	\draw[thick] (2,-5)--(15,-1);
	\draw[thick] (8,-10)--(17.5,-3);
    % dots
    \tdot{2}{-5}{blue}
    \node at (4.5,-6.8) {$\ddots$};
    \tdot{8}{-10}{blue}
    \node at (10, -11.5) {$\ddots$};
    \tdot{15}{-1}{blue}
    \tdot{17.5}{-3}{blue}
    % index
    \node at (2,1) {$1$};
    \node at (4,1) {$2$};
    \node at (6,1) {$\cdots$};
    \node at (8,1) {$j$};
    \node at (10,1) {$\cdots$};
    \node at (14,1) {$m+1$};
    \node at (18.5,1) {$m+2$};
    
    \node [left] at (0,-1) {$\pi_{m+1} = 1$};
    \node [left] at (0,-3) {$\pi_{m+2} = 2$};
    \node [left] at (0,-5) {$\pi_{1} = 3$};
    \node [left] at (0,-6.8) {$\pi_{2} = 4$};
    \node [left] at (-2,-8) {$\vdots$};
    \node [left] at (0,-10.2) {$\pi_{j} = j+2$};
    \node [left] at (-2,-11.5) {$\vdots$};
    \node [left] at (0,-13.8) {$\pi_{m} = m+2$};
	% comment
		% m
    \node[rotate = 0, red] at (6.5, -15.5) {$\underbrace{\hspace{2.5cm}}$};
    \node[red] at (6.5,-16.5) {$m$};
    	%n
    \node[rotate = 0, red] at (16.2, -15.5) {$\underbrace{\hspace{1cm}}$};
    \node[red] at (16.2,-16.5) {$n=2$};
	% comment
    \node at (10,-19) {(A)};
    
    %bipartite graph S'
	\begin{scope}[shift={(32,0)}]
	% edges
	\draw[thick] (2,-5)--(17.5,-3);
	\draw[thick] (8,-10)--(15,-1);
    % dots
    \tdot{2}{-5}{blue}
    \node at (4.5,-6.8) {$\ddots$};
    \tdot{8}{-10}{blue}
    \node at (10, -11.5) {$\ddots$};
    \tdot{15}{-1}{blue}
    \tdot{17.5}{-3}{blue}
    % index
    \node at (2,1) {$1$};
    \node at (4,1) {$2$};
    \node at (6,1) {$\cdots$};
    \node at (8,1) {$j$};
    \node at (10,1) {$\cdots$};
    \node at (14,1) {$m+1$};
    \node at (18.5,1) {$m+2$};
    
    \node [left] at (0,-1) {$\pi_{m+1} = 1$};
    \node [left] at (0,-3) {$\pi_{m+2} = 2$};
    \node [left] at (0,-5) {$\pi_{1} = 3$};
    \node [left] at (0,-6.8) {$\pi_{2} = 4$};
    \node [left] at (-2,-8) {$\vdots$};
    \node [left] at (0,-10.2) {$\pi_{j} = j+2$};
    \node [left] at (-2,-11.5) {$\vdots$};
    \node [left] at (0,-13.8) {$\pi_{m} = m+2$};
	% comment
		% m
    \node[rotate = 0, red] at (6.5, -15.5) {$\underbrace{\hspace{2.5cm}}$};
    \node[red] at (6.5,-16.5) {$m$};
    	%n
    \node[rotate = 0, red] at (16.2, -15.5) {$\underbrace{\hspace{1cm}}$};
    \node[red] at (16.2,-16.5) {$n=2$};
	% comment
    \node at (10,-19) {(B)};
	\end{scope}
\end{tikzpicture}
\caption{The two $1$-subgraphs in $\Sub{\bipart{m}{2}}$ with one edge incident to $1$ and another incident to $j \in [2,m]$: (A) the subgraph $S_1$ with non-crossing edges $(1,m+1)$ and $(j,m+2)$; (B) the subgraph $S_2$ with crossing edges $(1,m+2)$ and $(j,m+1)$.\label{fig:twoedges_twopatterns}}
\end{figure}

We will see that which subgraph is valid in fact depends on the parity of $m+j$. 
Lemma~\ref{lem:twoedges_1_j_valid} implies that there are exactly $m-1$ valid $1$-subgraphs of $G_{\bipart{m}{2}}$ with two edges, exactly one of which is incident to $1$. To show the lemma, we consider the two cases from Lemmas~\ref{lem:twoedges_noncrossing} and \ref{lem:twoedges_crossing}.

\begin{lemma}\label{lem:twoedges_noncrossing}
For $j\in[2,m]$, let $S_1$ be the $1$-subgraph of $G_{\bipart{m}{2}}$ with the two non-crossing edges $(1,m+1)$ and $(j, m+2)$ (see Figure~\ref{fig:twoedges_twopatterns}, Graph~(A)). Then $S_1$ is valid if and only if $m+j$ is even.
\end{lemma}

\begin{proof}

As in Figure~\ref{fig:twoedges_twopatterns}, Graph~(A), there are only two edges $(1,m+1)$ and $(j,m+2)$, so car $1=\pi_{m+1}$ and car $3=\pi_{1}$ prefer the same spot $1$, i.e., $p_1 = p_3 = 1$, and car $2=\pi_{m+2}$ and car $j+2=\pi_{j}$ prefer the same spot $j$, i.e., $p_2 = p_{j+2} = j$. Since no vertices from spot $1$ to spot $m$ have left edges, the preferences of the corresponding cars are just those spots. That is, we have $p_{\pi_{k}}=p_{k+2}=k$ for any $k \in \{1, \cdots, m\}$.

As such, using the map $\SubtoPF$ we get the parking function $p=\SubtoPF(S) = (1,j,1,2,\cdots,m)$. We now wish to determine under which conditions $S$ is valid, i.e.\ under which conditions $\OMVP{m+2}{p} = \bipart{m}{2}$. For this, we run the MVP parking process on $p$, which is illustrated on Table~\ref{tab:twoedges_noncrossing} below. 

\begin{table}[htbp]
\centering
\begin{tabular}{r|c|c|c|c|c|c|c|c|c|c|c|c|c|c}
\hline
\text{spot:} & 1 & 2 & 3 & 4 & 5 & $\cdots$ & j-2 & j-1 & j & j+1 & j+2 & $\cdots$ & m+1 & m+2 \\
\hline
$p_{\pi_{m+1}}=p_1=1:$ & 1 & \_ & \_ & \_ & \_ & $\cdots$ & \_ & \_ & \_ & \_ & \_ & $\cdots$ & \_ & \_ \\
$p_{\pi_{m+2}}=p_2=j:$ & 1 & \_ & \_ & \_ & \_ & $\cdots$ & \_ & \_ & 2 & \_ & \_ & $\cdots$ & \_ & \_ \\
$p_{\pi_{1}}=p_3=1:$ & 3 & 1 & \_ & \_ & \_ & $\cdots$ & \_ & \_ & 2 & \_ & \_ & $\cdots$ & \_ & \_ \\
$p_{\pi_{2}}=p_4=2:$ & 3 & 4 & 1 & \_ & \_ & $\cdots$ & \_ & \_ & 2 & \_ & \_ & $\cdots$ & \_ & \_ \\
$p_{\pi_{3}}=p_5=3:$ &  3 & 4 & 5 & 1 & \_ & $\cdots$ & \_ & \_ & 2 & \_ & \_ & $\cdots$ & \_ & \_ \\
 & & & & & & & $\vdots$ & & & & & & &  \\
$p_{\pi_{j-2}}=p_j=j-2:$ & 3 & 4 & 5 & 6 & 7 & $\cdots$ & j & \textcolor{blue}{1} & \textcolor{blue}{2} & \_ & \_ & $\cdots$ & \_ & \_ \\
$p_{\pi_{j-1}}=p_{j+1}=j-1:$ & 3 & 4 & 5 & 6 & 7 & $\cdots$ & j & j+1 & \textcolor{red}{2} & \textcolor{red}{1} & \_ & $\cdots$ & \_ & \_ \\
$p_{\pi_{j}}=p_{j+2}=j:$ & 3 & 4 & 5 & 6 & 7 & $\cdots$ & j & j+1 & j+2 & \textcolor{blue}{1} & \textcolor{blue}{2} & $\cdots$ & \_ & \_  \\
 & & & & &  & & $\vdots$ & &  &  &  &  &  &
\end{tabular}
\caption{The MVP parking process with $p=(1,j,1,2,3,4,\cdots,m)$.
\label{tab:twoedges_noncrossing}}
\end{table}

%$$\begin{array}{rcccccccccccccc}

%\end{array}
%$$

Since $p_1=1$, and $p_2 = j$, cars $1$ and $2$ initially park in spots $1$ and $j$ respectively. Then recall that for $k \geq 3$, we have $p_k = k-2$. As such, car $3$ parks in $1$ and bumps car $1$ into spot $2$. This is followed by car $4$ bumping car $1$ into spot $3$, car $5$ bumping it into spot $4$, and so on. Finally, after car $j$ parks, we see that exactly the first $j$ spots are all occupied, with the ``partial'' outcome $345 \cdots j12$.

Then car $j+1$ arrives, preferring spot $j-1$, so car $1$ will be bumped from spot $j-1$ to the first available spot $j+1$, yielding the updated partial outcome $345 \cdots j(j+1)21$. This is followed by car $j+2$ bumping car $2$ from spot $j$ into the first unoccupied spot $j+2$, yielding the partial outcome $345 \cdots j(j+1)(j+2)12$. We therefore see that each car $k > j$ will bump either car $1$ or car $2$, and flip their position in the partial outcome.

Since after car $j$ had parked, cars $1$ and $2$ were parked in that order (i.e.\ the ``correct'' order), this implies that we will reach the correct final outcome $\bipart{m}{2} = 345 \cdots (m+2)12$ (i.e.\ $S$ will be valid) if, and only if, the number of cars parking after car $j$ (i.e.\ the number of flips in the partial outcomes) is even. Since there are $m+2-j$ cars arriving after car $j$, this means that $S$ is valid if, and only if, $m+2-j$ is even, which is equivalent to $m+j$ being even, as desired.
\end{proof}

\begin{lemma}\label{lem:twoedges_crossing}
For $j\in[2,m]$, let $S_2$ be the $1$-subgraph of $G_{\bipart{m}{2}}$ with the two crossing edges $(1,m+2)$ and $(j, m+1)$ (see Figure~\ref{fig:twoedges_twopatterns}, Graph~(B)). Then $S_2$ is valid if and only if $m+j$ is odd.
\end{lemma}

\begin{proof}
As in Figure~\ref{fig:twoedges_twopatterns}, Graph~(B), there are only two edges $(1,m+2)$ and $(j,m+1)$, so car $2=\pi_{m+2}$ and car $3=\pi_{1}$ prefer the same spot $1$, i.e., $p_2 = p_3 = 1$, and car $1=\pi_{m+1}$ and car $j+2=\pi_{j}$ prefer the same spot $j$, i.e., $p_1 = p_{j+2} = j$. Since no vertices from spot $1$ to spot $m$ have left edges, the preferences of the corresponding cars are just those spots, i.e., $p_{\pi_{k}}=p_{k+2}=k$ for all $k \in \{1, \cdots, m\}$. 

As such, using the map $\SubtoPF$ we get the parking function $p=\SubtoPF(S) = (j,1,1,2,3,\cdots,m)$. We now wish to determine under which conditions $S$ is valid, i.e.\ under which conditions $\OMVP{m+2}{p} = \bipart{m}{2}$. From here, we proceed as in the proof of Lemma~\ref{lem:twoedges_noncrossing}. Because the proof is essentially analogous, we will be much briefer in our arguments here. We first run the MVP parking process on $p$, which is illustrated on Table~\ref{tab:twoedges_crossing} below.

\begin{table}[htbp]
\centering
\begin{tabular}{r|c|c|c|c|c|c|c|c|c|c|c|c|c|c}
\hline
\text{spot:} & 1 & 2 & 3 & 4 & 5 & $\cdots$ & j-2 & j-1 & j & j+1 & j+2 & $\cdots$ & m+1 & m+2 \\
\hline
$p_{\pi_{m+1}}=p_1=j:$ & \_ & \_ & \_ & \_ & \_ & $\cdots$ & \_ & \_ & 1 & \_ & \_ & $\cdots$ & \_ & \_ \\
$p_{\pi_{m+2}}=p_2=1:$ & 2 & \_ & \_ & \_ & \_ & $\cdots$ & \_ & \_ & 1 & \_ & \_ & $\cdots$ & \_ & \_ \\
$p_{\pi_{1}}=p_3=1:$ & 3 & 2 & \_ & \_ & \_ & $\cdots$ & \_ & \_ & 1 & \_ & \_ & $\cdots$ & \_ & \_ \\
$p_{\pi_{2}}=p_4=2:$ & 3 & 4 & 2 & \_ & \_ & $\cdots$ & \_ & \_ & 1 & \_ & \_ & $\cdots$ & \_ & \_ \\
$p_{\pi_{3}}=p_5=3:$ &  3 & 4 & 5 & 2 & \_ & $\cdots$ & \_ & \_ & 1 & \_ & \_ & $\cdots$ & \_ & \_ \\
 & & & & & & & $\vdots$ & & & & & & &  \\
$p_{\pi_{j-2}}=p_j=j-2:$ & 3 & 4 & 5 & 6 & 7 & $\cdots$ & j & \textcolor{red}{2} & \textcolor{red}{1} & \_ & \_ & $\cdots$ & \_ & \_ \\
$p_{\pi_{j-1}}=p_{j+1}=j-1:$ & 3 & 4 & 5 & 6 & 7 & $\cdots$ & j & j+1 & \textcolor{blue}{1} & \textcolor{blue}{2} & \_ & $\cdots$ & \_ & \_ \\
$p_{\pi_{j}}=p_{j+2}=j:$ & 3 & 4 & 5 & 6 & 7 & $\cdots$ & j & j+1 & j+2 & \textcolor{red}{2} & \textcolor{red}{1} & $\cdots$ & \_ & \_  \\
 &  &  &  &  &  &  & $\vdots$ & &  &  &  &  &  &
\end{tabular}
\caption{The MVP parking process for $p=(j,1,1,2,3,\cdots,m)$.
\label{tab:twoedges_crossing}}
\end{table}

%$\begin{array}{rcccccccccccccc}

%\end{array}
%$

We see that after the first $j$ cars have all parked, the first $j$ spots are all occupied. This time however, the partial outcome at this point is $345 \cdots j21$, i.e.\ the positions of cars $1$ and $2$ are reversed compared to the desired final outcome $\bipart{m}{2}$. As in the previous case, every car arriving subsequently will park just before cars $1$ and $2$, and flip their order. As such, the final outcome will have $1$, $2$ in that order if, and only if, an odd number of cars arrives after car $j$ (i.e.\ there are an odd number of flips). Since there are $m+2-j$ cars arriving after car $j$, this means that $S$ is valid if, and only if, $m+2-j$ is odd, which is equivalent to $m+j$ being odd, as desired.

\end{proof}

Lemmas~\ref{lem:twoedges_noncrossing} and \ref{lem:twoedges_crossing} imply Lemma~\ref{lem:twoedges_1_j_valid}. In particular, this shows that there are exactly $m-1$ valid $1$-subgraphs of $G_{\bipart{m}{2}}$ with two edges, exactly one of which is incident to the vertex $1$. We now characterise valid $1$-subgraphs with two edges which are both incident to $1$.

\begin{lemma}\label{lem:twoedges_1st_m+1_m+2}
For $m \geq 1$, let $S \in \Sub{\bipart{m}{2}}$ be the $1$-subgraph with edges $(1,m+1)$ and $(1,m+2)$ (see Figure~\ref{fig:twoedges_1st_m+1_m+2}). Then $S$ is valid if and only if $m$ is odd.
\end{lemma}

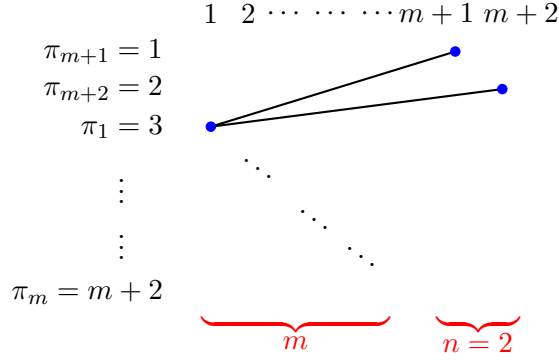
\begin{figure}[ht]
\centering
\begin{tikzpicture}[scale=0.25]
	%bipartite graph S 
	% edges
	\draw[thick] (2,-5)--(17.5,-3);
	\draw[thick] (2,-5)--(15,-1);
    % dots
    \tdot{2}{-5}{blue}
    \node at (4.5,-6.8) {$\ddots$};
    %\tdot{8}{-10}{blue}
    \node at (7.5, -9.5) {$\ddots$};
    \node at (10, -11.5) {$\ddots$};
    \tdot{15}{-1}{blue}
    \tdot{17.5}{-3}{blue}
    % index
    \node at (2,1) {$1$};
    \node at (4,1) {$2$};
    \node at (6,1) {$\cdots$};
    \node at (8.5,1) {$\cdots$};
    \node at (11,1) {$\cdots$};
    \node at (14,1) {$m+1$};
    \node at (18.5,1) {$m+2$};
    
    \node [left] at (0,-1) {$\pi_{m+1} = 1$};
    \node [left] at (0,-3) {$\pi_{m+2} = 2$};
    \node [left] at (0,-5) {$\pi_{1} = 3$};
    \node [left] at (-2,-8) {$\vdots$};
    \node [left] at (-2,-11) {$\vdots$};
    \node [left] at (0,-13.8) {$\pi_{m} = m+2$};
	% comment
		% m
    \node[rotate = 0, red] at (6.5, -15.5) {$\underbrace{\hspace{2.5cm}}$};
    \node[red] at (6.5,-16.5) {$m$};
    	%n
    \node[rotate = 0, red] at (16.2, -15.5) {$\underbrace{\hspace{1.1cm}}$};
    \node[red] at (16.2,-16.5) {$n=2$};
\end{tikzpicture}
\caption{The $1$-subgraph $S \in \Sub{\bipart{m}{2}}$ with edges $(1,m+1)$ and $(1,m+2)$.\label{fig:twoedges_1st_m+1_m+2}}
\end{figure}

\begin{proof}
As in Figure~\ref{fig:twoedges_1st_m+1_m+2}, there are two edges $(1,m+1)$ and $(1,m+2)$, so car $1=\pi_{m+1}$ and car $2=\pi_{m+2}$ both prefer the spot $1$, i.e.\ $p_1=p_2=1$. Since no vertices from spot $1$ to spot $m$ have left edges, the preferences of the corresponding cars are just those spots, i.e., $p_{\pi_{k}}=p_{k+2}=k$ for $1 \leq k \leq m$. 

Hence, using the map $\SubtoPF$ we get the parking function $p=\SubtoPF(S)=(1,1,1,2,3,\cdots,m)$. We now wish to determine under which conditions $S$ is valid, i.e.\ under which conditions $\OMVP{m+2}{p} = \bipart{m}{2}$. The proof proceeds along analogous lines to those of Lemmas~\ref{lem:twoedges_noncrossing} and \ref{lem:twoedges_crossing}, by running the MVP parking process on $p$, which is illustrated on Table~\ref{tab:twoedgesto1} below.

\begin{table}[htbp]
\centering
\begin{tabular}{r|c|c|c|c|c|c|c|c|c|c}
\hline
\text{spot:} & 1 & 2 & 3 & 4 & 5 & 6 & $\cdots$ & $\cdots$ & m+1 & m+2 \\
\hline
$p_{\pi_{m+1}}=p_1=1:$ & 1 & \_ & \_ & \_ & \_ & \_ & $\cdots$ & $\cdots$ & \_ & \_ \\
$p_{\pi_{m+2}}=p_2=1:$ & 2 & 1 & \_ & \_ & \_ & \_ & $\cdots$ & $\cdots$ & \_ & \_ \\
$p_{\pi_{1}}=p_3=1:$ & 3 & \textcolor{blue}{1} & \textcolor{blue}{2} & \_ & \_ & \_ & $\cdots$ & $\cdots$ & \_ & \_ \\
$p_{\pi_{2}}=p_4=2:$ & 3 & 4 & \textcolor{red}{2} & \textcolor{red}{1} & \_ & \_ & $\cdots$ & $\cdots$ & \_ & \_  \\
$p_{\pi_{3}}=p_5=3:$ &  3 & 4 & 5 & \textcolor{blue}{1} & \textcolor{blue}{2} & \_ & $\cdots$ & $\cdots$ & \_ & \_ \\
$\vdots$ &  &  &  &  &  &  & $\vdots$ &  &
\end{tabular}
\caption{The MVP parking process with $p=(1,1,1,2,3,\cdots,m)$. \label{tab:twoedgesto1}}
\end{table}

%$\begin{array}{rcccccccccccccc}

%\end{array}
%$

We see that after the first two cars park we have the partial outcome $21$, and every subsequent car arriving will flip the order of cars $1$ and $2$. This implies that we will reach the correct final outcome $\bipart{m}{2}$ (i.e.\ $S$ will be valid) if, and only if, the number of cars parking after car $2$ is odd (i.e.\ there are an odd number of flips after that point). Since there are $m+2-2=m$ cars arriving after car $2$, this means that $S$ is valid if, and only if, $m$ is odd, as desired.
\end{proof}

We are now equipped with all the necessary ingredients to prove Theorem~\ref{thm:enum_m_2}.

\begin{proof}[Proof of Theorem~\ref{thm:enum_m_2}]

We proceed by induction on $m \geq 0$. For $m = 0$, we have $\pi = \bipart{m}{2} = 12$, which has no inversions. As such, the only valid $1$-subgraph is the empty graph with no edges, so $\vert \FibMVP{2}{12} \vert = 1 = 0 + 1 + \lfloor \frac{(0+1)^2}{2} \rfloor$, as desired. For clarity, we also explicitly consider the case $m=1$. In that case, we have $\pi = \bipart{m}{2} = 312$. There are four $1$-subgraphs of $G_{\pi}$, which are all valid by Theorem~\ref{thm:MVPPF_perm_patterns} (see Example~\ref{ex:valid}). Therefore $\vert \FibMVP{3}{312} \vert = 4 = 1+1 + \lfloor \frac{(1+1)^2}{2} \rfloor$, as desired.

Suppose now that the formula holds for $m-1$ for some $m \geq 1$. We wish to count valid $1$-subgraphs of $G_{\bipart{m}{2}}$. Lemma~\ref{lem:addfirst_valid} tells us that the number of these where $1$ is isolated is equal to $\vert \FibMVP{m+1}{\bipart{m-1}{2}} \vert$. It remains to count those where $1$ is not isolated. There are two possibilities: the first where the subgraph has a single edge (incident to $1$), and the second where it has two edges.

For the first case, any $1$-subgraph consisting of a single edge is valid (see Proposition~\ref{pro:lower_bound} and following remarks). There are therefore two valid $1$-subgraphs of $G_{\bipart{m}{2}}$ with a single edge which is incident to $1$, corresponding to the edges $(1, m+1)$ or $(1, m+2)$ (see Figure~\ref{fig:insert_twomore_valid}).

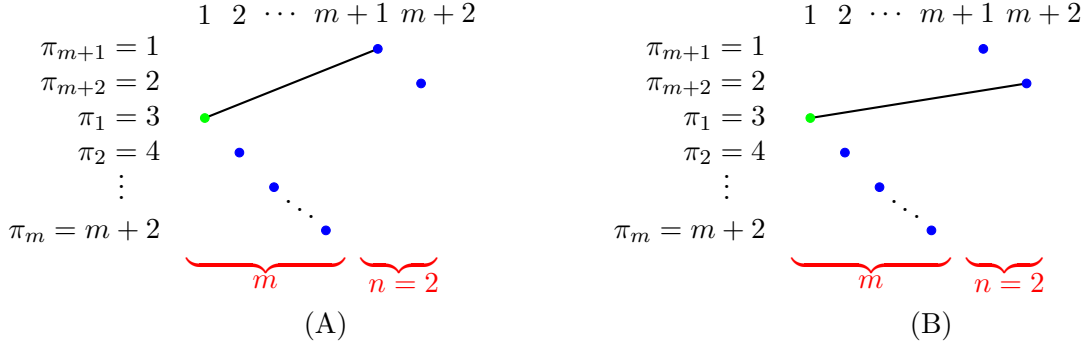
\begin{figure}[ht]
\centering
\begin{tikzpicture}[scale=0.23]
    
    %bipartite graph S with edge (1,m+1)
	% edges
	\draw[thick] (2,-5)--(12,-1);
	% dots
    \tdot{2}{-5}{green}
    \tdot{4}{-7}{blue}
    \tdot{6}{-9}{blue}
    \node at (7.5, -9.8) {$\ddots$};
    \tdot{9}{-11.5}{blue}
    \tdot{12}{-1}{blue}
    \tdot{14.5}{-3}{blue}
    % index
    \node at (2,1) {$1$};
    \node at (4,1) {$2$};
    \node at (6.5,1) {$\cdots$};
    %\node at (8,1) {$\cdots$};
    \node at (10.5,1) {$m+1$};
    \node at (15.5,1) {$m+2$};
    
    \node [left] at (0,-1) {$\pi_{m+1} = 1$};
    \node [left] at (0,-3) {$\pi_{m+2} = 2$};
    \node [left] at (0,-5) {$\pi_{1} = 3$};
    \node [left] at (0,-7) {$\pi_{2} = 4$};
    \node [left] at (-2,-8.5) {$\vdots$};
    \node [left] at (0,-11.5) {$\pi_{m} = m+2$};
	% comment
		% m
    \node[rotate = 0, red] at (5.5, -13.5) {$\underbrace{\hspace{2.1cm}}$};
    \node[red] at (5.5,-14.5) {$m$};
    	%n
    \node[rotate = 0, red] at (13.2, -13.5) {$\underbrace{\hspace{1cm}}$};
    \node[red] at (13.5,-14.5) {$n=2$};
    % comment
    \node at (9,-17) {(A)};
	
	%bipartite graph S with edge (1,m+2)
	\begin{scope}[shift={(35,0)}]
	% edges
	\draw[thick] (2,-5)--(14.5,-3);
	% dots
    \tdot{2}{-5}{green}
    \tdot{4}{-7}{blue}
    \tdot{6}{-9}{blue}
    \node at (7.5, -9.8) {$\ddots$};
    \tdot{9}{-11.5}{blue}
    \tdot{12}{-1}{blue}
    \tdot{14.5}{-3}{blue}
    % index
    \node at (2,1) {$1$};
    \node at (4,1) {$2$};
    \node at (6.5,1) {$\cdots$};
    %\node at (8,1) {$\cdots$};
    \node at (10.5,1) {$m+1$};
    \node at (15.5,1) {$m+2$};
    
    \node [left] at (0,-1) {$\pi_{m+1} = 1$};
    \node [left] at (0,-3) {$\pi_{m+2} = 2$};
    \node [left] at (0,-5) {$\pi_{1} = 3$};
    \node [left] at (0,-7) {$\pi_{2} = 4$};
    \node [left] at (-2,-8.5) {$\vdots$};
    \node [left] at (0,-11.5) {$\pi_{m} = m+2$};
	% comment
		% m
    \node[rotate = 0, red] at (5.5, -13.5) {$\underbrace{\hspace{2.1cm}}$};
    \node[red] at (5.5,-14.5) {$m$};
    	%n
    \node[rotate = 0, red] at (13.2, -13.5) {$\underbrace{\hspace{1cm}}$};
    \node[red] at (13.5,-14.5) {$n=2$};
    % commentB
    \node at (9,-17) {(B)};
	\end{scope}
\end{tikzpicture}
\caption{The two subgraphs of $\Sub{\bipart{m}{2}}$ containing a single edge incident to $1$ and no other edges: (A) the subgraph with edge $(1,m+1)$; (B) the subgraph with edge $(1,m+2)$. Both subgraphs are valid.\label{fig:insert_twomore_valid}}
\end{figure}

It therefore remains to count the number of valid $1$-subgraphs of $G_{\bipart{m}{n}}$ with two edges, at least one of which is incident to $1$. As explained, Lemma~\ref{lem:twoedges_1_j_valid} implies that there are exactly $m-1$ such $1$-subgraphs with exactly one edge incident to $1$. Finally, Lemma~\ref{lem:twoedges_1st_m+1_m+2} says that the $1$-subgraph with two edges incident to $1$ is valid if, and only if, $m$ is odd.

Bringing this together and applying the induction hypothesis, we get 
$$ \vert \FibMVP{m+2}{\bipart{m}{2}} \vert = 
\vert \FibMVP{m+1}{\bipart{m-1}{2}} \vert + 2 + (m-1) + \varepsilon_m = m + \lfloor \frac{m^2}{2} \rfloor + m + 1 + \varepsilon_m,$$
where $\varepsilon_m = 1$ if $m$ is odd, and $\varepsilon_m = 0$ if $m$ is even. It is then straightforward to obtain the desired formula.

\end{proof}

%%%%%%%%%%%%%%%%%%%%%%%%%%%%
%%%%%%%%%SECTION 5%%%%%%%%%%%%%
%%%%%%%%%%%%%%%%%%%%%%%%%%%%
\section{Interpreting the MVP outcome through the Abelian sandpile model}\label{sec:asm_prelims}

In this section we give an alternate interpretation of the MVP outcome map $\OOMVP{}$ in terms of the so-called \emph{Abelian sandpile model} (ASM). The ASM is a dynamic process on a graph. It was first introduced by Bak, Tang and Wiesenfeld~\cite{BTW} as an example of a process exhibiting the phenomenon known as \emph{self-organised criticality}. Later, Dhar~\cite{Dhar1990} formalised and named the model. We begin by introducing the model and important related concepts. In this work, we will only be concerned by the connection between the ASM and parking functions, which occurs on the complete graph $K_n$ (this is the graph with vertex set $[n]$ and one edge between any pair of (distinct) vertices). As such, we restrict ourselves to this setting here, essentially ignoring the geometry of the underlying graph.

\subsection{The Abelian sandpile model on complete graphs}\label{subsec:ASM}

A (sandpile) \emph{configuration} is a vector $c = (c_1, \cdots, c_n) \in \Zp^{n}$ which assigns a non-negative integer to each vertex. 
We think of $c_i$ as denoting the number of grains of sand at vertex $i$.
We denote $\Config{n}$ the set of sandpile configurations with $n$ vertices, called $n$-configurations for short. 
For $i \in [n]$, let $\alpha^i \in \Config{n}$ be the configuration such that $\alpha^i_i = 1$ and $\alpha^i_j = 0$ for all $j \neq i$.

Given a configuration $c$ and a vertex $i \in [n]$, if $c_i < n$, then the vertex $i$ is said to be \emph{stable} in the configuration $c$. Otherwise the vertex $i$ is \emph{unstable}. If all vertices in a configuration are stable, this configuration is stable, and we denote the set of all stable $n$-configurations as $\Stable{n}$. 
Unstable vertices topple as follows. 
For a configuration $c \in \Config{n}$ and a vertex $i \in [n]$ which is unstable in $c$, we define the \emph{toppling operator} at vertex $i$, denoted $\Top_i$, by:
\begin{equation}\label{eq:toppling}
\Top_i(c) := c - n \cdot \alpha^i + \sum_{j \neq i} \alpha^j, 
\end{equation}
where the addition operator on configurations denotes pointwise addition at each vertex. In words, the toppling of a vertex $i$ sends one grain of sand from $i$ to each of the remaining $(n-1)$ vertices $j$, and an additional grain exits the system.

Other vertices may become unstable after performing this toppling once, and we topple these also. 
Because whenever we topple an unstable vertex, one grain of sand exits the system, it is straightforward to show that, starting from an unstable configuration $c$ and toppling successively unstable vertices, we will eventually reach a stable configuration $c'$. We say that a sequence $S := v_1, \cdots, v_k$ of vertices is a \emph{toppling sequence} for $c$ if, for any $j < k$, the vertex $v_{j+1}$ is unstable in the configuration $\Top_{v_j} \cdots \Top_{v_1}(c)$, and $\Top_{v_k} \cdots \Top_{v_1}(c) \in \Stable{n}$. 
In words, starting from the configuration $c$, we can topple vertices of $S$ in order, and obtain a stable configuration after toppling them all.

Dhar showed (see e.g.~\cite[Section~5.2]{Dhar}) that all toppling sequences are equivalent up to permutation of the vertices, and that the stable configuration $c'$ reached after toppling all vertices in a toppling sequence $S$ does not depend on the order of $S$.
We call this $c'$ the \emph{stabilisation} of $c$, and denote it $\Stab(c)$.

We now define a Markov chain on the set $\Stable{n}$ of stable configurations.
Fix a probability distribution $\mu=(\mu_i)_{i \in [n]}$ on $[n]$ such that $\mu_i>0$ for all $i \in [n]$.
At each step of the Markov chain we add a grain at the vertex $i$ with probability $\mu_i$ and stabilise the resulting configuration.

The \emph{recurrent} configurations are those appear infinitely often in the long-time running of this Markov chain. We denote $\Rec{n}$ the set of recurrent $n$-configurations. The study of the recurrent configurations has been of central importance in ASM research (see e.g.~\cite{SelWheel} and references therein). Here we recall a classical characterisation of recurrent configurations: the so-called \emph{burning algorithm} due to Dhar~\cite[Section~6.1]{Dhar}, which provides a simple algorithmic process to check if a given configuration is recurrent or not.

\begin{theorem}\label{thm:characterisation_rec_configs}
Let $n \geq 1$, and $c \in \Stable{n}$ be a stable configuration. We define $\tilde{c} := c + \sum\limits_{i \in [n]} \alpha^i$ to be the configuration obtained by adding one grain to each vertex in $c$. Then $c$ is recurrent if, and only if, there exists a toppling sequence $S$ for $\tilde{c}$ in which each vertex of $[n]$ appears exactly once. Moreover, in this case, we have $\Stab \left( \tilde{c} \right) = c$.
\end{theorem}

There is a natural partial order on the set $\Rec{n}$ of recurrent configurations. For two configurations $c, c' \in \Rec{n}$, we define $c \preceq c'$ if, and only if, $c_i \leq c'_i$ for all $i \in [n]$. A \emph{minimal recurrent configuration} is a recurrent configuration which is minimal for this partial order. In words, a minimal recurrent configuration is a recurrent configuration where the removal of one grain of sand from any vertex would cause the configuration to no longer be recurrent. We denote $\MinRec{n}$ the set of minimal recurrent $n$-configurations. The following result appears in various parts of the literature (see e.g.~\cite{Sel}).

\begin{proposition}\label{pro:minrec_perm}
Let $n \geq 1$, and $c \in \Config{n}$ be a configuration. Then $c$ is minimal recurrent if, and only if, $c$ is a permutation of the set $\{0, \cdots, n-1\}$.
\end{proposition}

Given a minimal recurrent $n$-configuration $c \in \MinRec{n}$, we define the \emph{canonical toppling} of $c$ to be the permutation $\pi = \pi_1 \cdots \pi_n \in S_n$ where for each $i$, $\pi_i$ is the unique index $j$ such that $c_j = n-i$. We denote this permutation $\CanTop{c}$.

\begin{example}\label{ex:canon_top}
Consider the minimal recurrent configuration $c = (2, 4, 3, 0, 1) \in \MinRec{5}$. We calculate $\pi := \CanTop{c}$. By construction $\pi_1$ is the vertex with $5-1 = 4$ grains, i.e.\ $\pi_1 = 2$, $\pi_2$ is the vertex with $3$ grains, i.e.\ $\pi_2 = 3$, and so on. Finally we get $\CanTop{c} = 23154$.
\end{example}

The terminology \emph{canonical toppling} comes from the following observation. For $c \in \Rec{n}$, Dhar's burning criterion says that there must be a toppling sequence $S = \pi_1, \cdots, \pi_n$ for $\tilde{c}$ in which each vertex of $[n]$ appears exactly once, i.e.\ $S$ can be thought of as a permutation of $[n]$. It is straightforward to see that $\CanTop{c}$ is in fact the only possible toppling sequence for $\tilde{c}$ if $c$ is minimal recurrent, given Proposition~\ref{pro:minrec_perm}.

\subsection{The connection between ASM and MVP parking functions}\label{subsec:ASM_MVP}

In this section we will link the ASM to the MVP outcome map. We first recall the bijection between recurrent configurations and parking functions from the seminal work by Cori and Rossin~\cite{CR}.

\begin{theorem}\label{thm:bij_rec_pf}
Let $n \geq 1$. For a configuration $c = (c_1,\cdots,c_n) \in \Config{n}$, define $p = (p_1,\cdots,p_n) := (n-c_1,\cdots,n-c_n)$ to be the n-complement of $c$ (we write $p = n-c$ for short). Then $c$ is recurrent if, and only if, $p$ is a (MVP) parking function. Thus the map $c \mapsto n-c$ defines a bijection from $\Rec{n}$ to $\MVP{n}$ ($= \PF{n}$).
\end{theorem}

Our main goal of this section is an interpretation of the MVP outcome of a given parking function $p$ through the recurrent configuration corresponding to $p$ via the bijection above. We begin by describing an algorithm which reduces a recurrent configuration $c$ to a minimal recurrent configuration $c' \preceq c$.

\begin{algorithm}\label{algo:c-minrec}
Given an input recurrent configuration $c$, we modify $c$ as follows. In each iteration, we look for the first pair of duplicate values in $c$.
\begin{enumerate}
\item
\begin{itemize}
\item If $c$ contains no duplicate values, return $c$.
\item Otherwise, define $j:= \min\{j' \in [n]; c_{j'} \in \{c_1,\cdots,c_{j'-1}\}\}$, to be the index of the first duplicate value encountered in $c$, and $i$ to be the index s.t.\ $i<j$ and $c_i=c_j$.
\end{itemize}
\item While there exists $j' \in \{1, \cdots, j\} \setminus \{ i \}$ such that $c_{j'} = c_i$ (i.e.\ $c_i$ is a duplicate value in $\{c_1, \cdots, c_j\}$), decrease the value of $c_i$ by one ($c_i=c_i-1$).
\end{enumerate}

Repeat Step $(1)$ and Step $(2)$ until $c$ is returned.
\end{algorithm}

\begin{remark}\label{rem:algo_terminates}
Algorithm~\ref{algo:c-minrec} necessarily terminates as each time we return to step $(1)$, $j$ strictly increases. Moreover, by construction, the algorithm returns a configuration $c$ such that $c_i \neq c_j$ for all $i \neq j$.
\end{remark}

\begin{example}\label{ex:algo}
Consider the recurrent configuration
$c=(11, 9, 5, 8, 1, 9, 4, 8, 4, 9, 10, 0)$. We wish to determine the output of Algorithm~\ref{algo:c-minrec}.

\vspace{1em}
\textbf{Iteration $1$}:
$\begin{array}{lllllllllllll}
\text{index:} & 1 & \textcolor{blue}{2} & 3 & 4 & 5 & \textcolor{blue}{6} & 7 & 8 & 9 & \textcolor{blue}{10} & 11 & 12 \\
\text{values:} & 11, & \textcolor{red}{9}, & 5, & 8, & 1, & \textcolor{red}{9}, & 4, & 8, & 4, & \textcolor{red}{9}, & 10, & 0 \\
\end{array}
$

Starting from index $1$, the value ``$9$'' (in red) is the first duplicate value encountered (it is the value at index $2$, $6$ and $10$, which are in blue) with this value. 
\begin{itemize}
\item In Step~(1), we set $j$ to be the index of the first duplicate, i.e.\ $j = 6$, and $i$ to be the index where that value was previously encountered, i.e.\ $i = 2$.
\item In Step~(2), we decrement $c_i = c_2$ until it is no longer a value encountered elsewhere to the left of $j = 6$. In this case the values $9$ and $8$ are already present (highlighted in pink) at indices $6$ and $4$ respectively, so we finally set $c_2 = 7$.

\vspace{0.5em}
$\begin{array}{llllllllllllllll}
\text{index:} & 1 & \textcolor{blue}{2} & 3 & 4 & 5 & 6 & 7 & 8 & 9 & 10 & 11 & 12\\
\text{values:} & 11, & \textcolor{red}{7}, & 5, & \colorbox{pink}{8}, & 1, & \colorbox{pink}{9}, & 4, & 8, & 4, & 9, & 10, & 0 \\
\end{array}
$
\end{itemize}

\medskip

\textbf{Iteration $2$}:
$\begin{array}{lllllllllllll}
\text{index:} & 1 & 2 & 3 & \textcolor{blue}{4} & 5 & 6 & 7 & \textcolor{blue}{8} & 9 & 10 & 11 & 12 \\
\text{values:} & 11, & 7, & 5, & \textcolor{red}{8}, & 1, & 9, & 4, & \textcolor{red}{8}, & 4, & 9, & 10, & 0 \\
\end{array}
$

Starting from index $1$, the value ``$8$'' (in red) is the first duplicate value encountered.
\begin{itemize}
\item In Step~(1), we set $j$ to be the index of the first duplicate, i.e.\ $j = 8$, and $i$ to be the index where that value was previously encountered, i.e.\ $i = 4$.
\item In Step~(2), we decrement $c_i = c_4$ until it is no longer a value encountered elsewhere to the left of $j = 8$. In this case the values $8$ and $7$ are already present (highlighted in pink) at indices $8$ and $2$ respectively, so we finally set $c_4 = 6$.

\vspace{0.5em}$\begin{array}{lllllllllllll}
\text{index:} & 1 & 2 & 3 & \textcolor{blue}{4} & 5 & 6 & 7 & 8 & 9 & 10 & 11 & 12 \\
\text{values:} & 11, & \colorbox{pink}{7}, & 5, & \textcolor{red}{6}, & 1, & 9, & 4, & \colorbox{pink}{8}, & 4, & 9, & 10, & 0 \\
\end{array}
$
\end{itemize}

\medskip

\textbf{Iteration $3$}:
$\begin{array}{lllllllllllll}
\text{index:} & 1 & 2 & 3 & 4 & 5 & 6 & \textcolor{blue}{7} & 8 & \textcolor{blue}{9} & 10 & 11 & 12 \\
\text{values:} & 11, & 7, & 5, & 6, & 1, & 9, & \textcolor{red}{4}, & 8, & \textcolor{red}{4}, & 9, & 10, & 0 \\
\end{array}
$

Starting from index $1$, the value ``$4$'' (in red) is the first duplicate value encountered.
\begin{itemize}
\item In Step~(1), we set $j$ to be the index of the first duplicate, i.e.\ $j = 9$, and $i$ to be the index where that value was previously encountered, i.e.\ $i = 7$.
\item In Step~(2), we decrement $c_i = c_7$ until it is no longer a value encountered elsewhere to the left of $j = 9$. In this case only the value $4$ is already present (highlighted in pink) at index $9$, so we finally set $c_7 = 3$.

\vspace{0.5em}
$\begin{array}{lllllllllllll}
\text{index:} & 1 & 2 & 3 & 4 & 5 & 6 & \textcolor{blue}{7} & 8 & 9 & 10 & 11 & 12 \\
\text{values:} & 11, & 7, & 5, & 6, & 1, & 9, & \textcolor{red}{3}, & 8, & \colorbox{pink}{4}, & 9, & 10, & 0 \\
\end{array}
$
\end{itemize}

\medskip

\textbf{Iteration $4$}:
$\begin{array}{lllllllllllll}
\text{index:} & 1 & 2 & 3 & 4 & 5 & \textcolor{blue}{6} & 7 & 8 & 9 & \textcolor{blue}{10} & 11 & 12 \\
\text{values:} & 11, & 7, & 5, & 6, & 1, & \textcolor{red}{9}, & 3, & 8, & 4, & \textcolor{red}{9}, & 10, & 0 \\
\end{array}
$

Starting from index $1$, the value ``$9$'' (in red) is the first duplicate value encountered.
\begin{itemize}
\item In Step~(1), we set $j$ to be the index of the first duplicate, i.e.\ $j = 10$, and $i$ to be the index where that value was previously encountered, i.e.\ $i = 6$.
\item In Step~(2), we decrement $c_i = c_6$ until it is no longer a value encountered elsewhere to the left of $j = 10$. This time the values $9$ (index $10$), $8$ (index $8$), $7$ (index $2$), $6$ (index $4$), $5$ (index $3$), $4$ (index $9$), and $3$ (index $7$) all already appear, so we finally set $c_6 = 2$.

\vspace{0.5em}
$\begin{array}{lllllllllllll}
\text{index:} & 1 & 2 & 3 & 4 & 5 & \textcolor{blue}{6} & 7 & 8 & 9 & 10 & 11 & 12 \\
\text{values:} & 11, & \colorbox{pink}{7}, & \colorbox{pink}{5}, & \colorbox{pink}{6}, & 1, & \textcolor{red}{2}, & \colorbox{pink}{3}, & \colorbox{pink}{8}, & \colorbox{pink}{4}, & \colorbox{pink}{9}, & 10, & 0 \\
\end{array}
$
\end{itemize}

\medskip

\textbf{Iteration $5$}:
$\begin{array}{lllllllllllll}
\text{index:} & 1 & 2 & 3 & 4 & 5 & 6 & 7 & 8 & 9 & 10 & 11 & 12 \\
\text{values:} & 11, & 7, & 5, & 6, & 1, & 2, & 3, & 8, & 4, & 9, & 10, & \phantom{,}0 \\
\end{array}
$

At this point we have reached a configuration $c$ with no duplicate values. Hence, Algorithm~\ref{algo:c-minrec} outputs the final configuration $c=(11, 7, 5, 6, 1, 2, 3, 8, 4, 9, 10, 0)$.

\end{example}

We may notice on the above example that  the final configuration output is in fact a permutation of the set $\{0, \cdots, 11\}$, i.e.\ a minimal recurrent $12$-configuration. This turns out to be true in general, and is one of the main results of this section.

\begin{proposition}\label{pro:c_single_minrec}
Given an input $c \in \Rec{n}$, the output of Algorithm~\ref{algo:c-minrec} is a permutation of the set $\{0,\cdots,n-1\}$, and thus a minimal recurrent $n$-configuration. We denote it $\minrec{c}$. As such, we may view the algorithm as a map $\mathrm{minrec} : \Rec{n} \rightarrow \MinRec{n}$.
\end{proposition}

\begin{proof}
Given that by construction the output configuration has no duplicate values, by Proposition~\ref{pro:minrec_perm} it is sufficient to show that the algorithm always outputs a recurrent configuration. In fact, it is sufficient to show that recurrence is preserved through any single decrement on Step~(2). In particular, it suffices to show that if $c$ is a recurrent configuration, and $i \neq j \in [n]$ are such that $c_i = c_j$ (i.e.\ a duplicated value in $c$), then the configuration $c' := c - \alpha^i$, resulting from removing one single grain of sand from $i$, is also recurrent.

To show this, we first apply Dhar's burning algorithm (Theorem~\ref{thm:characterisation_rec_configs}) to $c$. Since $c$ is recurrent, there exists a toppling sequence $S = v_1,\cdots,v_n$ for $\tilde{c} := c + \sum\limits_{i \in [n]} \alpha^i$. Without loss of generality we may assume that $j$ appears before $i$ in $S$. Otherwise we simply swap them, and this does not affect the ability of any vertex to topple in $S$. Furthermore, we may assume that $i$ immediately follows $j$ in $S$. Indeed, if this is not the case, we move $i$ to immediately follow $j$, yielding a new sequence $S' = v_1, \cdots, v_k = j, v_{k+1} = i, \cdots, v_n$. Since initially $c_i = c_j$, and $i$ and $j$ receive the same number of grains through toppling $v_1, \cdots, v_{k-1}$, we see that after toppling $v_1, \cdots, v_{k-1}$, $i$ and $j$ still have the same number of grains. This implies that they are both unstable at that point (since $j$ has to be in order for $S$ to be a toppling sequence). Moreover, subsequent vertices at most receive one extra grain in $S'$ compared to $S$, which does not affect their capacity to topple. Therefore $S'$ is also a toppling sequence for $\tilde{c}$.

We claim that $S'$ is also a valid toppling sequence for $\tilde{c'}$. Indeed, given that the only difference between $c$ and $c'$ is at the vertex $i$, it is sufficient to show that, starting from the configuration $\tilde{c'}$, the vertex $i$ is unstable after toppling $v_1, \cdots, v_k = j$ in $S'$ (the capacity of other vertices to topple will be the same in $\tilde{c'}$ as in $\tilde{c}$). But by definition we have $c'_i = c_i - 1 = c_j - 1$. Moreover, after toppling $v_1, \cdots, v_{k-1}$, the vertex $j$ must be unstable. In particular, at that point the vertex $i$ requires at most extra grain to in turn become unstable, and it receives that grain when vertex $v_k = j$ topples. Thus $S'$ is also a toppling sequence for $\tilde{c'}$, which implies that $c'$ is recurrent by Theorem~\ref{thm:characterisation_rec_configs}. This completes the proof.
\end{proof}

We are now ready to state the main result of this section.

\begin{theorem}\label{thm:minrec_outcome_CanonTopp}
Given an input $c \in \Rec{n}$, Algorithm~\ref{algo:c-minrec} yields an output $\minrec{c} \in \MinRec{n}$. Moreover, if $p(c) = n-c$ is the parking function corresponding to $c$ via the bijection of Cori and Rossin (Theorem~\ref{thm:bij_rec_pf}), then we have $\OMVP{n}{p(c)} = \CanTop{\minrec{c}}$.
\end{theorem}

In words, given a parking function $p \in \MVP{n}$, we can describe its outcome through applying Algorithm~\ref{algo:c-minrec} to the corresponding recurrent configuration $c = n-p \in \Rec{n}$, and taking the canonical toppling of the output minimal recurrent configuration $\minrec{c}$. To prove Theorem~\ref{thm:minrec_outcome_CanonTopp}, we need to prove two things: that the MVP outcome of a parking function $p$ is unchanged through applying Algorithm~\ref{algo:c-minrec} to the corresponding recurrent configuration $c = n-p$ (Lemma~\ref{lem:algo_outcome_unchanged}), and that the canonical toppling does indeed give the correct MVP outcome for a minimal recurrent configuration (Lemma~\ref{lem:minrec_CanonTopp_outcome}).

\begin{lemma}\label{lem:algo_outcome_unchanged}
Let $c \in \Rec{n}$ be a recurrent configuration, and $c'$ be the recurrent configuration obtained from applying a single iteration of Algorithm~\ref{algo:c-minrec} (Steps~(1) and (2)) to $c$. Let $p = n-c$ and $p' = n-c'$ be the MVP parking functions corresponding to $c$ and $c'$ respectively. Then we have $\OMVP{n}{p'}=\OMVP{n}{p}$.
\end{lemma}

\begin{proof}
% statement for c->c'->p'->outcome(p')
Let $c \in \Rec{n}$ be a recurrent configuration, and $i,j \in [n]$ as in Step~(1) of Algorithm~\ref{algo:c-minrec}. In the corresponding MVP parking function, this means that we have $p_i=p_j=n-c_i=n-c_j$, and moreover that $j$ is the first such $j$ satisfying this equality, which means that it is the first car to produce a collision in the MVP parking process.

Now let $c'$, $p$, $p'$ be as in the statement of the lemma. That is, $c'_i$ is the maximal \emph{value} less than or equal to $c_j$ such that no other \emph{vertex} $j' \leq j$ has the same \emph{number of grains} (i.e.\ $c'_i \neq c_{j'}$ for all $j' \leq j$). In the corresponding parking function $p'$, this implies that $p'_i$ is the minimal \emph{spot} greater than or equal to $p_j$ such that no other \emph{car} $j' \leq j$ has the same \emph{preference} (i.e. $p'_i \neq p_{j'}$ for all $j' \leq j$).

In other words, in the MVP parking process for $p$, $p'_i$ is exactly the spot that car $i$ will move to when it is bumped out of its original preference $p_i$ by car $j$. Also note that this is the only difference between the MVP parking functions $p$ and $p'$. It is then straightforward to see that the outcome of the entire MVP parking process will be the same whether car $i$ initially prefers spot $p_i$ and is first bumped to $p'_i$ by car $j$ (as in the parking process for $p$), or whether car $i$ initially prefers spot $p'_i$ directly without needing to be bumped there (as in the parking process for $p'$). This implies that $p$ and $p'$ have the same MVP outcome, as desired.
\end{proof}

\begin{lemma}\label{lem:minrec_CanonTopp_outcome}
If $c \in \MinRec{n}$ is a minimal recurrent configuration, and $p = p(c) := n - c$ is the corresponding MVP parking function, then we have $\OMVP{n}{p(c)} = \CanTop{c}$.
\end{lemma}

\begin{proof}
Since $c$ is minimal recurrent, it is a permutation of $\{0, \cdots, n-1\}$. 
By construction, the first vertex recorded in the canonical toppling sequence $\CanTop{c}$ is the unique $k \in [n]$ such that $c_k = n-1$. In terms of the parking function, this means that $k$ is the unique $k \in [n]$ such that $p_k = 1$, i.e.\ the car that ends up in spot $1$. In other words, the first vertex recorded in $\CanTop{c}$ is the first car recorded in $\OMVP{n}{p(c)}$. Iterating this reasoning on the remaining vertices/cars immediately yields the desired result.
\end{proof}

\begin{remark}\label{rem:classical_PF_ASM}
Since MVP parking functions are also classical parking functions, we may ask how the outcome map $\OOPF{}$ for classical parking functions translates to the ASM through the Cori-Rossin bijection. In fact, this can be done in analogous fashion, with a minor modification to Algorithm~\ref{algo:c-minrec}. All that is needed is to replace the decrement of $c_i$ in Step~(2) with the decrement of $c_j$ (in fact, in this alternate version there is no need to define the index $i$ at all). Indeed, in the proof of Lemma~\ref{lem:algo_outcome_unchanged} we saw that the decrement of $c_i$ in Step~(2) corresponds to car $i$ being bumped out of its preferred spot by car $j$. In the classical case, this is replaced by car $j$ driving on to the first available spot (instead of car $i$), so we decrease $c_j$ instead in the corresponding recurrent configuration.
\end{remark}

\section{Discussion and future work}\label{sec:future}

In this paper, we have investigated the outcome fibres of MVP parking functions. We have represented the fibre of a given outcome permutation $\pi$ as certain \emph{valid} subgraphs of the inversion graph $G_{\pi}$ of $\pi$. In turn, this produced new and improved upper and lower bounds on the fibre size $\left\vert \FibMVP{}{\pi} \right\vert$. It remains an open problem to fully characterise which $1$-subgraphs of $G_{\pi}$ are valid, or equivalently which are invalid, only in terms of the subgraphs (without needing to check by running the MVP parking process).

In~\cite[Theorem~3.2]{HarrisMVP} it was shown that there exist invalid $1$-subgraphs if, and only if, the permutation $\pi$ contains at least one of the patterns $321$ and $3412$. In the case of the pattern $321$, Proposition~\ref{pro:upper_bound} gives a general characterisation of a ``forbidden motif'' that would always render a subgraph $S$ invalid in terms of the existence of a directed path $\overrightarrow{P_2}$ in the subgraph $S$. However, this is not the only forbidden motif. Indeed, Example~\ref{ex:invalid} gives the example of a subgraph $S \in \Sub{321}$ with edges $(1,2)$ and $(1,3)$ which is invalid. But this particular motif does not translate to the general case, as detailed in Remark~\ref{rem:matching_not_general}.

In the case of the pattern $3412$, the situation is also complex. Indeed, in Section~\ref{sec:m_n} we fully characterised the valid $1$-subgraphs of the complete bipartite grah $K_{m,n}$. However, as we saw in Lemmas~\ref{lem:twoedges_noncrossing} and \ref{lem:twoedges_crossing} which subgraphs are valid depends on some parity conditions. As such, it seems difficult to hope for a general expression of forbidden motifs in $1$-subgraphs of $G_{\pi}$ when $\pi$ contains the pattern $3412$.

More generally, we noted in Theorem~\ref{thm:MVPPF_perm_patterns} that permutations avoiding the patterns $321$ and $3412$ are exactly those whose inversion graphs are acyclic. However, the arguments for this statement in both this paper and in the previous work by Harris \emph{et al.}~\cite{HarrisMVP} are quite \emph{ad hoc}, relying on a case-by-case analysis of how the MVP parking process behaves in those cases. It seems plausible that there is some sort of ``meta'', more general, explanation for why the appearance of a cycle in the inversion graph forces the existence of invalid $1$-subgraphs, but we have been unable to find such an explanation.

One might also ask how much improvement our upper and lower bounds on fibre sizes give, and how close these bounds are to the actual fibre size. In the general case this depends quite strongly on the permutation $\pi$. Here we provide a table in the case where $\pi = \dec{n}$ is the decreasing permutation (i.e.\ $G_{\pi} = K_n$) for the first few values of $n$. In this table, we calculate the total number of $1$-subgraphs (the original upper bound from~\cite[Theorem~3.1]{HarrisMVP}), the number of $\overrightarrow{P_2}$-free $1$-subgraphs (our new upper bound), the number of valid subgraphs (i.e.\ the fibre size), and the number of HS $1$-subgraphs (our new lower bound).

\begin{table}[ht]
\begin{tabular}{c|c|c|c|c}
$n$ & $1$-subgraphs & $\overrightarrow{P_2}$-free & valid & HS \\
\hline
1 & 1 & 1 & 1 & 1 \\
2 & 2 & 2 & 2 & 2 \\
3 & 6 & 5 & 4 & 4 \\
4 & 24 & 15 & 9 & 8 \\
5 & 120 & 52 & 21 & 16 \\
6 & 720 & 203 & 51 & 32 \\
7 & 5040 & 877 & 127 & 64 \\
8 & 40320 & 4140 & 323 & 128 \\
9 & 362880 & 21147 & 835 & 256 \\
\end{tabular}
\caption{Comparing our new bounds to the actual fibre size in the case of the decreasing permutation $\dec{n}$.\label{table:bound_comparison}}
\end{table}

We can see that, while neither bound appears particularly tight as $n$ grows, we do get quite an improvement on the previous upper bound by imposing the $\overrightarrow{P_2}$-free condition. The somewhat attentive reader will have noticed that the number of HS $1$-subgraphs of $G_{\dec{n}}$ appears to be $2^{n-1}$. This can be established as follows. Given a HS subgraph $S$, and an edge $e = (i, j) \in S$ (with $i < j$), we define $P_e := \{i, i+1, \cdots, j-1\} \subseteq [n-1]$. It is reasonably straightforward to check that the map $S \mapsto \bigcup\limits_{e \in S} P_e$ is a bijection between HS $1$-subgraphs of $G_{\dec{n}}$ and subsets of $[n-1]$.

The highly attentive and eagle-eyed reader will have recognised that the $\overrightarrow{P2}$-free numbers $1, 2, 5, \cdots$ correspond to the well-studied \emph{Bell numbers} given by Sequence~A000110 in the OEIS~\cite{OEIS}. This can be established by giving a bijection between set partitions of $[n]$, which are counted by the Bell numbers (see the OEIS entry), and $\overrightarrow{P2}$-free subgraphs of $G_{\dec{n}}$, $\mathcal{P} \mapsto S$, as follows. For each part $P$ in a set partition $\mathcal{P}$ of $[n]$, we put edges in $S$ between the minimal element of $P$ and all the other elements of $P$ (if $P$ is reduced to a single element $i$, then $i$ is an isolated vertex in $S$). By construction, this gives a $\overrightarrow{P_2}$-free $1$-subgraph $S$, and it is reasonably straightforward to check that this construction is bijective.

We should note that the above bijections on $\overrightarrow{P_2}$-free and HS $1$-subgraphs can be extended to the setting of general permutations, with additional restrictions imposed by the geometry of the inversion graph. For example, in the general setting, the $\overrightarrow{P_2}$-free subgraphs map bijectively to partitions of $[n]$ where, in each part, the minimal element is incident to all other elements.

Another avenue of research explored in this paper concerns the study of certain subsets of MVP parking functions. In Section~\ref{sec:MotzPF} we focused on Motzkin parking functions, where each spot is preferred by at most two cars. In particular, if we restrict ourselves to MVP parking functions whose outcome is the decreasing permutation $\dec{n} := n(n-1) \cdots 1$, we get bijective correspondences with Motzkin paths and non-crossing matching arc diagrams.

Then, in Section~\ref{sec:m_n}, we studied the permutations in the fibre of the permutation $\bipart{m}{n} : = (n+1)(n+2) \cdots (n+m)12\cdots n$ whose corresponding inversion graph is the complete bipartite graph $K_{m,n}$. In particular, via a careful analysis of which $1$-subgraphs are valid, we were able to obtain an explicit enumeration for the fibre size in the case $n=2$. For more general values of $n$, Table~\ref{table:general_m_n} gives the numbers of MVP parking function whose outcome is $\bipart{m}{n}$.

\begin{table}[ht]

\begin{tabular}{c*{7}{|c}}
\diagbox{n}{m} & 1 & 2 & 3 & 4 & 5 & 6 & 7 \\
\hline
1 & 2 & 3 & 4 & 5 & 6 & 7 & 8 \\
2 & 4 & 7 & 12 & 17 & 24 & 31 & 40 \\
3 & 8 & 16 & 30 & 50 & 77 & 110 & 155 \\
4 & 16 & 36 & 70 & 130 & 220 & 341 & 512 \\
5 & 32 & 80 & 161 & 315 & 577 & 967 & 1532 \\
6 & 64 & 176 & 369 & 738 & 1425 & 2560 & 4281 \\
7 & 128 & 384 & 840 & 1706 & 3392 & 6431 & 11337 \\
\end{tabular}
\caption{The complete bipartite fibre sizes $ \left\vert \FibMVP{m+2}{\bipart{m}{n}} \right\vert $ for $1 \leq m,n \leq 7$.\label{table:general_m_n}}
\end{table}

The sequence in the second column for $m=2$ appears to correspond to Sequence~A045891 in the OEIS~\cite{OEIS}, and we leave it to enterprising readers to try to prove this. No other row or column matches any OEIS entry, and neither does the diagonal reading of the array.

One may also note that there are straightforward enumerative values in the cases $m=1$ and $n=1$. These correspond to the cases where the inversion graph is a \emph{star graph}, with a central vertex incident to all other vertices (and no other edges). In both cases, the graph is acyclic, so the fibre size is simply given by the product formula in Theorem~\ref{thm:MVPPF_perm_patterns}, and it is not difficult to get the direct enumerations of Table~\ref{table:general_m_n}. In fact, these cases were already shown as applications of the product formula by Harris \emph{et al.}~\cite[Section~3.1]{HarrisMVP}.

One possible direction of future research involves other families of regular permutations, or permutations where the inversion graph is regular. So far the following cases have been considered: complete graphs, complete bipartite graphs for $n=2$, and star graphs. Another family of regular graphs which are permutation inversion graphs are the \emph{complete split graphs}, and we briefly discuss this case.

For $m,n \geq 1$, the complete split graph $S_{m,n}$ consists of a \emph{clique} (complete subgraph) of size $n$ together with an \emph{independent set} (a set of vertices with no edges between them) of size $m$, with one edge between any vertex in the clique and any vertex in the independent set (see e.g.~\cite{DukesSplit}). There are two families of permutations whose inversion graphs are complete split graphs: the permutations $\splitright{m}{n} := (n+1) \cdots (n+m)n(n-1) \cdots 1$ and $\splitleft{m}{n} := (m+n)(m+n-1)\cdots (m+1)12\cdots m$. It would be interesting to study the MVP outcome fibres of these families of permutations. Note that if $m = 1$, we recover complete graphs and the permutations $\dec{n}$, and if $n=1$, we recover the star graphs discussed above.

Similarly to the complete bipartite case in Section~\ref{sec:m_n}, a good starting point may therefore be to first consider $m=2$ case. In this case the corresponding inversion graph is simply the complete graph with one of its edges removed. We first consider the permutation $\splitright{2}{n-2} = (n-1)n(n-2)(n-3) \cdots 21$ for $n \geq 3$, corresponding to the edge $(1,2)$ being removed in the inversion graph. Numerical evidence from Harris \emph{et al.}~\cite[Section~5]{HarrisMVP} suggests that this permutation may be the one which maximises the MVP fibre size, at least when $n \geq 6$. We conducted our own numerical experimentation to support this, and found that for $n = 7, 8, 9, 10, 11$, the permutation $\splitright{2}{n-2}$ above had a larger MVP fibre size than the decreasing permutation $\dec{n}$. Indeed, if anything, the gap between the fibre sizes seems to be growing, as shown in Table~\ref{table:dec_almostdec} below.

\begin{table}[ht]
\begin{tabular}{c*{9}{|c}}
$n$ & 3 & 4 & 5 & 6 & 7 & 8 & 9 & 10 & 11\\
\hline
$\vert \FibMVP{n}{\dec{n}} \vert $ & 4 & 9 & 21 & 51 & 127 & 323 & 835 & 2188 & 5798 \\
\hline
$\vert \FibMVP{n}{\splitright{2}{n-2}} \vert $ & 3 & 8 & 20 & 51 & 131 & 341 & 897 & 2383 & 6385 \\
\end{tabular}
\caption{Comparing fibre sizes for the decreasing and split permutations.\label{table:dec_almostdec}}
\end{table}

\begin{conjecture}\label{conj:largest_fibre}
For $n \geq 6$, the permutation with the largest MVP fibre size is the split permutation $\splitright{2}{n-2} := (n-1)n(n-2)(n-3) \cdots 21$.
\end{conjecture}

This may perhaps seem counter-intuitive given our subgraph representation. Indeed we might expect that the permutation with the largest fibre is the one whose inversion graph has the most edges (since more edges means more possible subgraphs). One would therefore expect this to be the decreasing permutation. However, as seen in Section~\ref{subsec:Motz_non_crossing_diagrams}, valid subgraphs in the decreasing case must be matchings. In particular, there cannot be two edges $(i,j)$ and $(i,k)$ with $i < j,k$. Numerical experimentation suggests that for the split permutation $\splitright{2}{n-2}$, we can have such combinations of edges, at least for $i=1,2$. For example, if $\pi = 3421$, the subgraph with edges $(1,3)$ and $(1,4)$ is valid (the parking function associated is $(1,1,1,2)$). It seems plausible that these additional valid subgraphs are more than enough to compensate for the loss of the edge $(1,2)$ compared to the decreasing case.

Now let us consider the split permutation $\splitleft{2}{n-2} := n(n-1)\cdots 312$, for $n \geq 3$. The corresponding inversion graph is the complete graph with edge $(n-1, n)$ removed. It turns out that in this case, the MVP fibres are enumerated by Motzkin numbers. We sketch a proof of this by exhibiting a bijection $\Psi : \Valid{\dec{n}} \rightarrow \Valid{\splitleft{2}{n-2}}$ for $n \geq 3$ as follows.
\begin{enumerate}
\item If $S \in \Valid{\dec{n}}$ does not contain the edge $(n-1, n)$, we simply set $\Psi(S) = S$ (as edge sets).
\item If $S \in \Valid{\dec{n}}$ contains the edge $(n-1, n)$, and the vertex $n-2$ is isolated, we set $\Psi(S) = S \setminus \{(n-1, n)\} \cup \{(n-2, n-1), (n-2, n)\}$.
\item If $S \in \Valid{\dec{n}}$ contains the edge $(n-1, n)$, and an edge $(i, n-2)$ for some $i$, let $S_i$ denote the set of edges with end-points strictly between $i$ and $n-2$. For $e = (j,k) \in S_i$, we define $e' := (j+1, k+1)$ to be $e$ ``shifted'' rightwards by one column, and $S'_i = \{e'; \, e \in S_i \}$. We then set $\Psi(S) = S \setminus \big( S_i \cup \{(i, n-2), (n-1, n)\} \big) \cup S'_i \cup \{(i, n-1), (i, n)\}$. In words, we replace the edges $(i,n-2)$ and $(n-1,n)$ with the edges $(i, n-1)$ and $(i, n)$, and shift all edges which are nested inside $(i, n-2)$ by one column to the right.
\end{enumerate}

\begin{example}\label{ex:dec_split}
We take $n = 8$, and consider the $1$-subgraph $S$ of $\dec{8}$ whose edges are $(2,6)$, $(3,4)$, and $(7,8)$. Here we are in Case~(3) above, so in the $1$-subgraph $S' = \Psi(S)$ we replace the edges $(2,6)$ and $(7,8)$ with edges $(2,7)$ and $(2,8)$, and shift the edge $(3,4)$ to the right, i.e.\ it becomes $(4,5)$. This is illustrated on Figure~\ref{fig:dec_neardec} below. Since $S$ is a non-crossing matching, we know that $S$ is valid by Theorem~\ref{thm:nonCrossing_dec}. To see that $S'$ is valid, we obtain its corresponding parking preference $p := \SubtoPF\left(S'\right) = (2, 2, 6, 4, 4, 3, 2, 1)$. We can then check that $\OMVP{8}{p} = 87654312 = \splitleft{2}{6}$, as desired.

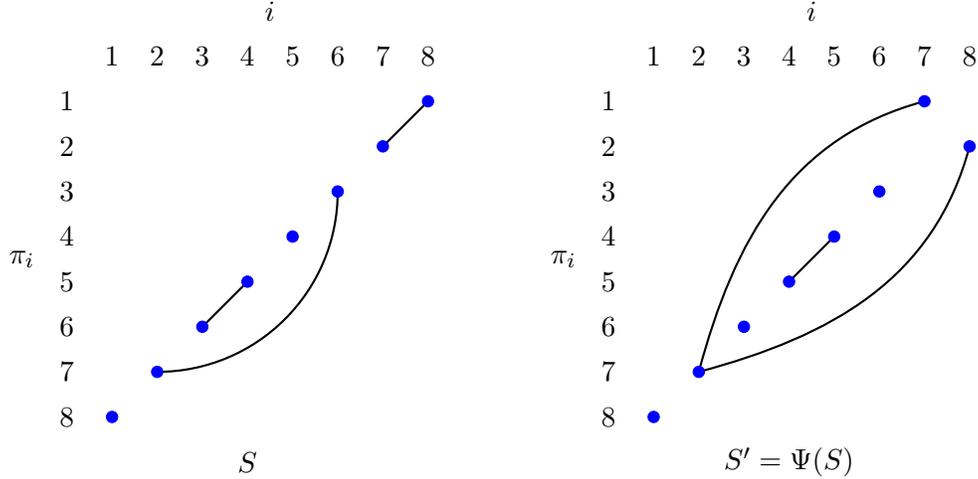
\begin{figure}[ht]
\centering
\begin{tikzpicture}[scale=0.3]
	% complete graph    
    % edges
    \draw [thick] (6,-12)--(8,-10);
    \draw [thick, out=0, in=-90] (4,-14) to (12,-6);
    \draw [thick] (14,-4)--(16,-2);
    % dots
    \foreach \x in {1,...,8}
	   \tdot{2*\x}{-18 + 2*\x}{blue};
    % values
    \foreach \x in {1,...,8}
      \node at (2*\x,0) {$\x$};
    % values
    \foreach \y in {1,...,8}
      \node at (0,-2*\y) {$\y$};
    % labels i, \pi_i, S
    \node at (9,2) {$i$};
    \node at (-2,-9) {$\pi_i$};
    \node at (8,-18) {$S$};
    
    \begin{scope}[shift={(24,0)}]
    % near-complete graph    
    % edges
    \draw [thick] (8,-10)--(10,-8);
    \draw [thick, out=15, in=-105] (4,-14) to (16,-4);
    \draw [thick, out=75, in=-165] (4,-14) to (14,-2);
    % dots
    \foreach \x in {1,...,6}
	   \tdot{2*\x}{-18 + 2*\x}{blue};
	 \tdot{14}{-2}{blue}
	 \tdot{16}{-4}{blue}
    % values
    \foreach \x in {1,...,8}
      \node at (2*\x,0) {$\x$};
    % values
    \foreach \y in {1,...,8}
      \node at (0,-2*\y) {$\y$};
    % labels i, \pi_i, S'
    \node at (9,2) {$i$};
    \node at (-2,-9) {$\pi_i$};
    \node at (8,-18) {$S' = \Psi(S)$};
    \end{scope}
\end{tikzpicture}
\caption{Illustrating the bijection $\Psi:\Valid{\dec{n}} \rightarrow \Valid{\splitleft{2}{n-2}}$ in Case~(3).\label{fig:dec_neardec}}
\end{figure}
\end{example}

We state the following theorem without proof. It can be shown through a careful analysis of valid $1$-subgraphs of $G_{\splitleft{2}{n-2}}$, similar to the proofs in Sections~\ref{subsec:Motz_non_crossing_diagrams} or \ref{sec:m_n}. The key is that in $S'$, the rightwards shift creates an empty space $i+1$ for car $1$ to park temporarily when it is bumped from spot $i$ by car $2$.

\begin{theorem}\label{thm:neardec_Motzkin}
For any $n \geq 3$, the map $\Psi : \Valid{\dec{n}} \rightarrow \Valid{\splitleft{2}{n-2}}$ is a bijection. In particular, we have $\left\vert \FibMVP{n}{\splitleft{2}{n-2}} \right\vert = \left\vert \FibMVP{n}{\dec{n}} \right\vert = \vert \Motz{n} \vert$.
\end{theorem}

Note that while the MVP fibres of $\splitleft{2}{n-2}$ are enumerated by Motzkin numbers, the parking functions in those fibres are not Motzkin parking functions as defined in Section~\ref{subsec:Motz_general}. Indeed, in the parking function $p = (2, 2, 6, 4, 4, 3, 2, 1) \in \FibMVP{8}{\splitleft{2}{6}} $ from Example~\ref{ex:dec_split}, the spot $2$ is preferred by three cars (equivalently the vertex $2$ has two right-edges in the corresponding subgraph), so $p$ is not a Motzkin parking function.

\section*{Acknowledgments}

The research leading to these results was partially supported by the National Natural Science Foundation of China (NSFC) under Grant Agreement No 12101505. The authors have no competing interests to declare that are relevant to the content of this article. The code that was used to generate the experimental data in this paper is available from the corresponding author on reasonable request.

\bibliographystyle{plain}
\bibliography{MVP_PF_sandpile_bibliography}

\begin{thebibliography}{10}

\bibitem{AdeYan}
Ayomikun Adeniran and Catherine Yan.
\newblock On increasing and invariant parking sequences.
\newblock {\em Australas. J. Comb.}, 79:167--182, 2021.

\bibitem{BTW}
Per Bak, Chao Tang, and Kurt Wiesenfeld.
\newblock Self-organized criticality: An explanation of the 1/f noise.
\newblock {\em Phys. Rev. Lett.}, 59:381--384, Jul 1987.

\bibitem{BarMotz}
E.~Barcucci, R.~Pinzani, and R.~Sprugnoli.
\newblock The {Motzkin} family.
\newblock {\em PU.M.A., Pure Math. Appl., Ser. A}, 2(3-4):249--279, 1991.

\bibitem{BV}
Robert Brignall and Vincent Vatter.
\newblock Labelled well-quasi-order for permutation classes.
\newblock {\em Comb. Theory}, 2(3):54, 2022.
\newblock Id/No 14.

\bibitem{CamDefPF}
{Peter Jephson} Cameron, Daniel Johannsen, Thomas Prellberg, and Pascal
  Schweitzer.
\newblock Counting defective parking functions.
\newblock {\em Electron. J. Comb.}, 15(1):research paper R92, 378--388, 2008.

\bibitem{ChevPyl}
Denis Chebikin and Pavlo Pylyavskyy.
\newblock A family of bijections between {{\(G\)}}-parking functions and
  spanning trees.
\newblock {\em J. Comb. Theory, Ser. A}, 110(1):31--41, 2005.

\bibitem{HarrisSeq2}
Douglas Chen, Pamela~E. Harris, J.~Carlos~Mart\'{i}nez Mori, Eric
  Pab\'{o}n-Cancel, and Gabriel Sargent.
\newblock Permutation invariant parking assortments.
\newblock {\em Enumer. Comb. Appl.}, 4(1):Paper No. S2R4, 2021.

\bibitem{HarrisNaples1}
Alex Christensen, Pamela~E. Harris, Zakiya Jones, Marissa Loving, Andr{\'e}s
  Ramos~Rodr{\'{\i}}guez, Joseph Rennie, and Gordon Rojas~Kirby.
\newblock A generalization of parking functions allowing backward movement.
\newblock {\em Electron. J. Comb.}, 27(1):research paper p1.33, 18, 2020.

\bibitem{HarrisNaples2}
Laura Colmenarejo, Pamela~E. Harris, Zakiya Jones, and Christo Keller.
\newblock Counting k-naples parking functions through permutations and the
  k-naples area statistic.
\newblock {\em Enumer. Comb. Appl.}, 1(2):Paper No. S2R11, 2021.

\bibitem{CoriPou}
Robert Cori and Dominique Poulalhon.
\newblock Enumeration of (p,q)-parking functions.
\newblock {\em Discrete Math.}, 256(3):609--623, 2002.
\newblock LaCIM 2000 Conference on Combinatorics, Computer Science and Appl
  ications.

\bibitem{CR}
Robert {Cori} and Dominique {Rossin}.
\newblock {On the sandpile group of dual graphs}.
\newblock {\em {Eur. J. Comb.}}, 21(4):447--459, 2000.

\bibitem{Dhar1990}
Deepak Dhar.
\newblock Self-organized critical state of sandpile automaton models.
\newblock {\em Phys. Rev. Lett.}, 64:1613--1616, Apr 1990.

\bibitem{Dhar}
Deepak Dhar.
\newblock Theoretical studies of self-organized criticality.
\newblock {\em Physica A}, 369(1):29--70, 2006.
\newblock Fundamental Problems in Statistical Physics.

\bibitem{DukesSplit}
Mark Dukes.
\newblock The sandpile model on the complete split graph, motzkin words, and
  tiered parking functions.
\newblock {\em J. Comb. Theory, Ser. A}, 180:105418, 2021.

\bibitem{EhrHapp}
Richard Ehrenborg and Alex Happ.
\newblock Parking cars of different sizes.
\newblock {\em Am. Math. Mon.}, 123(10):1045--1048, 2016.

\bibitem{HarrisSeq1}
Spencer Franks, Pamela~E. Harris, Kimberly Harry, Jan Kretschmann, and Megan
  Vance.
\newblock Counting parking sequences and parking assortments through
  permutations.
\newblock {\em Enumer. Comb. Appl.}, 4(1):Paper No. S2R2, 2021.

\bibitem{HarrisMVP}
Pamela~E. Harris, Brian~M. Kamau, J.~Carlos~Mart\'{i}nez Mori, and Roger Tian.
\newblock On the outcome map of mvp parking functions: Permutations avoiding
  321 and 3412, and motzkin paths.
\newblock {\em Enumer. Comb. Appl.}, 3(2):Paper No. S2R11, 16, 2023.

\bibitem{Khare2014}
Niraj Khare, Rudolph Lorentz, and Catherine~Huafei Yan.
\newblock Bivariate gon{\v{c}}arov polynomials and integer sequences.
\newblock {\em Science China Mathematics}, 57:1561--1578, 2014.

\bibitem{KonWeiss}
Alan~G. Konheim and Benjamin Weiss.
\newblock An occupancy discipline and applications.
\newblock {\em SIAM J. Appl. Math.}, 14(6):1266--1274, 1966.

\bibitem{PostPF}
Alexander Postnikov and Boris Shapiro.
\newblock Trees, parking functions, syzygies, and deformations of monomial
  ideals.
\newblock {\em T Am Math Soc}, 356:3109--3142, 2004.

\bibitem{Sel}
Thomas Selig.
\newblock The stochastic sandpile model on complete graphs, 2022.

\bibitem{SelWheel}
Thomas Selig.
\newblock Combinatorial aspects of sandpile models on wheel and fan graphs.
\newblock {\em Eur. J. Comb.}, 110:103663, 2023.

\bibitem{YanUpq}
Lauren Snider and Catherine Yan.
\newblock U-parking functions and (p,q)-parking functions.
\newblock {\em Advances in Applied Mathematics}, 134:102309, 2022.

\bibitem{Yan2023}
Lauren Snider and Catherine Yan.
\newblock ($\mathfrak{S}_p \times \mathfrak{S}_q$)-invariant graphical parking
  functions, 2023.

\bibitem{StanHyp}
Richard~P. Stanley.
\newblock Hyperplane arrangements, interval orders, and trees.
\newblock {\em Proc. Natl. Acad. Sci. USA}, 93(6):2620--2625, 1996.

\bibitem{StanleyEC}
Richard~P. Stanley and Sergey Fomin.
\newblock {\em Enumerative Combinatorics}, volume~2 of {\em Cambridge Studies
  in Advanced Mathematics}.
\newblock Cambridge University Press, 1999.

\bibitem{PitStan}
Richard~P. Stanley and Jim Pitman.
\newblock A polytope related to empirical distributions, plane trees, parking
  functions, and the associahedron.
\newblock {\em Discrete Comput. Geom.}, 27(4):603--634, 2002.

\bibitem{OEIS}
{The OEIS Foundation Inc.}
\newblock {The On-line Encyclopedia of Integer Sequences}, 2023.

\bibitem{Yan}
Catherine~H. Yan.
\newblock {\em Parking Functions}, chapter~13.
\newblock CRC Press, 2015.

\bibitem{Yin2023}
Mei Yin.
\newblock Parking functions, multi-shuffle, and asymptotic phenomena.
\newblock {\em La Matematica}, pages 1--25, 2023.

\end{thebibliography}

\end{document}